\documentclass[12pt, colorinlistoftodos]{amsart}
\usepackage{graphicx}
\usepackage{amsmath, enumerate}
\usepackage{amsfonts, amssymb, amsthm, xcolor, stmaryrd}
\usepackage{tikz, tikz-cd, caption, tabu, longtable, mathrsfs, xtab, centernot, mathtools}
\usepackage{dsfont}
\usepackage[textwidth=20mm]{todonotes}
\usepackage[]{hyperref}
\usepackage[style=alphabetic, backend=bibtex,doi=false,isbn=false,url=false,eprint=true, sorting=nyt, maxnames=99, backref=true]{biblatex}
\makeatletter 
\@mparswitchfalse
\makeatother
\normalmarginpar

\numberwithin{equation}{section}
\theoremstyle{plain}

\newtheorem{thm}{Theorem}[section]
\newtheorem{lemma}[thm]{Lemma}
\newtheorem{prop}[thm]{Proposition}
\newtheorem{cor}[thm]{Corollary}
\newtheorem{remark}[thm]{Remark}

\theoremstyle{definition}

\theoremstyle{remark}
\newtheorem{rmk}[thm]{Remark}

\newcommand{\Hb}{\mathbb{H}}
\newcommand{\SL}{{\mathrm{SL}}}
\newcommand{\Zb}{\mathbb{Z}}
\newcommand{\Cb}{\mathbb{C}}
\newcommand{\Qb}{\mathbb{Q}}

\newcommand{\Nb}{\mathbb{N}}
\newcommand{\F}{\mathbb{F}}
\newcommand{\lp}{\left (}
\newcommand{\rp}{\right )}
\newcommand{\pf}{\mathfrak{p}}
\newcommand{\Gal}{{\mathrm{Gal}}}

\newcommand{\Rb}{\mathbb{R}}

\newcommand{\smat}[4]{\left(\begin{smallmatrix}
                 #1 & #2\\
                 #3 & #4
\end{smallmatrix}\right)}
\newcommand{\pmat}[4]{\begin{pmatrix}
                 #1 & #2\\
                 #3 & #4
\end{pmatrix}}

\newcommand{\Oc}{\mathcal{O}}
\newcommand{\df}{\mathfrak{d}}

\newcommand{\af}{\mathfrak{a}}
\newcommand{\Nm}{{\mathrm{Nm }}}

\newcommand{\ord}{{\mathrm{ord}}}

\newcommand{\lf}{\mathfrak{l}}

\newcommand{\Ec}{{\mathcal{E}}}

\newcommand{\Sc}{{\mathcal{S}}}

\newcommand{\Z}{{\mathbb{Z}}}
\newcommand{\tr}
{\mathrm{tr}}

\newcommand{\Wb}{\mathbb{W}}
\newcommand{\Ab}{\mathbb{A}}

\def\ol#1{\overline{#1}}

 \newcommand{\Hom}{\mathrm{Hom}}
 \newcommand{\End}{\mathrm{End}}
  \newcommand{\cha}{\mathrm{Char}}
  
  \newcommand{\cp}{c^{(\mathfrak{p})}}
    \newcommand{\cpp}{c^{(\mathfrak{p}')}}

\newcommand{\Diff}{\mathrm{Diff}}
\newcommand{\Ep}{E}
\newcommand{\RC}{\mathrm{RC}}

 \author[E.~Assing]{Edgar Assing}
 \address{
Mathematical Institute of the University of Bonn,
 Endenicher Allee 60,
 D--53115 Bonn,
 Germany
 }
 \email{assing@math.uni-bonn.de}
 \author[Y.~Li]{Yingkun Li}
 \address{
Max Planck Institute for Mathematics,
    Vivatsgasse 7, 
    D--53111     Bonn,
    Germany}
\email{yingkun@mpim-bonn.mpg.de}

 \author[T.~Wang]{Tian Wang} \address{
Department of Mathematics \& Statistics,
Concordia University,
Montreal, Quebec, Canada
 }
 \email{tian.wang@concordia.ca}
 
 \author[J.~Xia]{Jiacheng Xia} 
 \address{
Department of Mathematics,
The University of Hong Kong,
Pokfulam, Hong Kong
 }
 \email{philimathmuse@gmail.com}
 \date{\today}

\begin{document}
\title{Isogenies of CM Elliptic Curves}

 \begin{abstract}
   Given two CM elliptic curves over a number field and a natural number $m$, 
   we establish a polynomial lower bound (in terms of $m$) for the number of rational primes $p$ such that the reductions of these elliptic curves modulo a prime above $p$ are $m$-isogenous. The proof relies on higher Green functions and theorems of Gross-Zagier and Gross-Kohnen-Zagier.
   A crucial observation is that the Fourier coefficients of incoherent Eisenstein series can be approximated by those of coherent Eisenstein series of increasing level.
   Another key ingredient is an explicit upper bound for the Petersson norm of an arbitrary elliptic modular form in terms of finitely many of its Fourier coefficients at the cusp infinity, which  is a result of independent interest.
\end{abstract}
\maketitle
% \makeatletter
% \providecommand\@dotsep{5}
 %\def\listtodoname{List of Todos}
 %\def\listoftodos{\@starttoc{tdo}\listtodoname}
 %\makeatother

\section{Introduction}\label{intro}
In the study of the arithmetic of elliptic curves, it is of great interest to understand isogenies of elliptic curves in analogy to understanding their torsion points. Let $E_1$ and $E_2$ be two elliptic curves defined over a number field $K$. Let $\pf$ be a prime ideal of the ring of integers $\Oc_K$ such that both $E_1$ and $E_2$ have good reductions at $\pf$. Let $E_{1, \pf}$ and $E_{2, \pf}$ denote these two reductions, respectively. If there is a cyclic $m$-isogeny between $E_1$ and $E_2$, then it is clear that there is also a cyclic $m$-isogeny between $E_{1, \pf}$ and $E_{2, \pf}$.  Conversely, one can ask for a fixed positive integer  $m$, if $E_{1, \pf}$ and $E_{2, \pf}$ are  $m$-isogenous for sufficiently many primes $\pf$,  does it imply that  $E_1$ and $E_2$ are also $m$-isogenous?   For a rational prime $\ell$, such a local-global principle for non-trivial $\ell$-torsion points of the Mordell-Weil group of an elliptic curve was investigated by Katz \cite{Ka1981}. Subsequently, the local-global principle for $\ell$-isogenies of elliptic curves has been extensively studied in works such as \cite{Anni2014, BaCr2014, Sutherland2012, Vo2020}.  In general, addressing this local-global principle is closely related to the broader question of understanding the image of the mod  $\ell$ Galois representation attached to elliptic curves. 

There is a significant distinction between cyclic isogenies in CM and non-CM cases. For instance, for any two non-CM elliptic curves over a number field, there is at most one number that can appear as degrees of cyclic isogenies between them. In contrast, there are infinitely many different degrees of cyclic isogenies between any two rationally isogenous CM elliptic curves. In terms of the local-global principle for cyclic $m$-isogenies, we want to understand how many primes are ``sufficiently many'' in the CM case, which leads us to the following natural question.

Suppose $E_1$ and $E_2$ are two CM elliptic curves over $K$ which are not $m$-isogenous. For a fixed positive integer $m$, how many prime ideals  in $\Oc_K$ are there such that the reductions $E_{1, \pf}$ and $E_{2, \pf}$ at $\pf$ become $m$-isogenous?  Note that, without fixing the degree $m$, similar questions have already been addressed in previous works. For instance, an upper bound of such primes with bounded norm  was established in \cite{CoHiWa2024, Wong2022} for non-isogenous   non-CM elliptic curves using analytic tools such as effective Chebotarev density theorem and effective joint Sato-Tate distribution. For non-isogenous  CM elliptic curves,   an asymptotic formula can be easily derived using Deuring's criterion  (see Remark \ref{rmk:upper-bound-pi}). For both CM and non-CM elliptic curves, Charles  \cite{Chrles2018} applied Arakelov intersection theory to prove the existence of infinitely many primes $\pf$ for which $E_{1, \pf}$ and $E_{2, \pf}$ are isogenous. Nonetheless, as far as we know, there is no such  results in the literature when we fix the degree of cyclic isogenies.   In this note,  we establish an polynomial lower bound, in terms of $m$, for the number of primes where the reductions of two CM elliptic curves become $m$-isogenous.

For $i = 1, 2$, let $E_i$ be fixed elliptic curves with CM by orders $\Oc_{D_i} := \Zb[(D_i + \sqrt{D_i})/2]$ in imaginary quadratic fields $K_i = \Qb(\sqrt{D_i})$.
They are defined over the ring class fields $H_{D_i}$ for $\Oc_{D_i}$.
The composite $H := H_{D_1} H_{D_2}$ is an abelian extension of $K := K_1K_2$, which is a biquadratic field and a quadratic CM extension of the real quadratic field $F = \Qb(\sqrt{D}), D := D_1D_2$ when $D$ is not a rational square.
%%%We will use $\Delta_L$ to denote the discriminant of a number field $L$, and set $\Delta_i := \Delta_{K_i}$. 

Let $p$ be a rational prime inert in $K_i$.
Then it splits completely in $H_{D_i}$, hence also in $H$.
%\todo[color=yellow]{TW: $p$ splits in $H_{D_i}$ iff $\left(\frac{D_i}{p} \right)=1$ iff $p$ splits in $K_i$. In this case, the reduction $E_{i, \frak{p}}$ is ordinary. But the bound of $\pi(m)$ does not change.\\YL: $p$ splits in $H_{D_i}$ iff  $\left(\frac{D_i}{p} \right)=-1$ iff $p$ is inert $K_i$.}
%\todo[color=yellow]{TW: Can we assume $\pf$ is a degree 1 prime? If so, we can assume $\pf$ is a prime in any finite extension of $H$.   If not, is it possible to get the degree of the prime?\\YL: yes, $\pf$ has degree 1.}
For a prime $\pf$ in $H$ above such $p$, we will denote $E_{i, \pf}$ the supersingular elliptic curve defined over $\Oc_H/\pf \cong \mathbb{F}_{p^2}$ obtained from $E_i$ by reduction modulo $\pf$.\footnote{It is an exercise to show that this rational prime $p$ splits into  two primes of the same inertia degree in $K$.  Let $\pf_
K$ be any  prime of $K$ above $p$. Then $\Nm(\pf_K)=p^2$ and it is obvious that $\pf_K$ splits completely in $H_iK$ for $i = 1, 2$. 
  As $H = H_1H_2 = (H_1K)(H_2K)$, $\pf_K$ also splits completely in $H$. So for a prime $\pf$ in $H$ lying above $p$ (hence above $\pf_K$), we have $\Nm(\pf)=p^2$.} 
Then they are isogenous over $\ol{\F}_{p}$.
For $m \in \Nb$, denote $E_{1, \pf}\sim_m E_{2, \pf}$ if there is a cyclic isogeny of degree $m$ between $E_{1, \pf}$ and $E_{2, \pf}$.

%\todo[color=yellow]{TW: $m$-isogeous over $\F_{\pf} := \Oc_H/\pf$.\\ YL: not necessarily.}
In this note, we will investigate the number of $p$ such that $E_{1, \pf}\sim_m E_{2, \pf}$ as a function of $m$, i.e.,  the size of  the set
\begin{equation}
  \label{eq:pim}
  \pi(m) := \pi_{E_1, E_2}(m) := \{ p \text{ prime}\colon p\nmid m,\,  E_{1, \pf} \sim_m E_{2, \pf}  \text{ for a prime }\pf \text{ over } p  \}. 
\end{equation}
%\todo{EA: I have removed primes dividing $m$ from the definition.\\TW: Thanks!}
This is a finite set with elements $p \in \pi(m)$ bounded by $D m^2/4$ (see Remark \ref{rmk:r-bound}). 
Furthermore, only primes $p$ that are non-split in $K_i$ could appear in $\pi(m)$, otherwise, $E_{i, \pf}$ would be ordinary and have different endomorphism rings. 
In \cite{Li21}, one of us showed that $\pi(m)$ is non-empty for any $m$.  
%In this note, we give some information about the degree of the isogeny. 
This is improved in the following theorem.

%%% \todo{YL: prove later for non-fundamental $D_i$.}
\begin{thm}
  \label{thm:main}
  Suppose $D_1, D_2$ are fundamental discriminants and co-prime. 
%  There exists $\delta > 0$
%  and constant $C = C(E_1, E_2) > 0$
%  such that
Then
  \begin{equation}
    \label{eq:bound}
    %% |\pi(m)| > C\cdot  \frac{m^\delta}{\log m}
        \#\pi(m) \gg_{E_1, E_2, \epsilon} m^{1/5 - \epsilon} 
  \end{equation}
%  for all $m \in \Nb$ and
  for any $\epsilon > 0$. 
\end{thm}
% \todo[color=yellow]{TW: Do you think we can better  average results for a family of CM elliptic curves ordered by the didscriminats $D_1, D_2$?\\ YL: good question, I am not sure. For this one needs dependence on $D_1, D_2$. This is explicit except the step when we applied equidistribution of Hecke orbit. Heuristically larger disc means more equidistribution. One would need to quantify this though. EA: Such a quantification of the equidistribution result does not seem to be in the literature after all. Probably one could work it out in our setting (the Clozel-Oh-Ullmo result we cite is much more general), but this would probably take some time.}
\begin{rmk}
  The condition on $D_i$ is technical, and can be removed with more work. Furthermore, Theorem \ref{thm:main} is effective as   the proof uses equidistribution of Hecke orbits \cite{COU01}.
  %%%This means that, at least ,
In principle,   one can make the constants explicit and obtain a quantitative dependence on the discriminants $D_1$ and $D_2$.
\end{rmk}
%\todo[inline]{EA: Is there a conjecture on the 'actual' size of $\pi(m)$?\\
% YL: I would expect $c \cdot m^2/log(m)$. EA: Thanks! (also for the explanation below.)}
%\todo[color=yellow]{YL: turn red text below into a remark.\\ TW: Feel free to revise or delete as needed.\\ YL: looks good.}
\begin{rmk}
  \label{rmk:upper-bound-pi}
We provide a rough estimate for an expected upper bound of $\#\pi(m)$. 
Assume the CM elliptic curves $E_1$ and $E_2$ are defined over $\mathbb{Q}$. Counting primes in $\pi(m)$ is almost equivalent to counting the number of distinct prime divisors of 
$
\varphi_m(j_{E_1}, j_{E_2})
$, where $\varphi_m(X, Y)\in \Z[X, Y]$ is the $m$-th modular polynomial and $j(E_i)$ is the $j$-invariant of $E_i$ (see also (\ref{eq:pim2})).
%%%\todo{JX: this fact is already stated in Eq. (1.3). We can merge these two parts.\\TW: (1.3) does not tell how large are the coeffiients of $\phi_m(X, Y)$. I added ``(see also (\ref{eq:pim2}))". But if you have better idea how to merge them, please feel free to do so. }
By a result on the heights of modular polynomials \cite{Co1984}, the number of prime divisors of $\varphi_m(j(E_1), j(E_2))$ is bounded above by $O\left(\log|\varphi_m(j(E_1), j(E_2))|\right)=O_{E_1, E_2}(m\log m\log\log m)$.
If $E_1$ and $E_2$ are  defined over a  general number field $K$, the $O$-term also depends on $K$. 

Additionally, we note that applying Deuring's criterion alone does not provide meaningful information about the distribution of primes in $\pi(m)$. For CM elliptic curves  $E_1$ and $E_2$, Deuring's criterion implies that the density of rational primes $p$ lying below some prime  $\mathfrak{p}$ of $H_{D_i}$, for which $E_{\mathfrak{p}}$ is  supersingular, is 1/2. Assuming that $D_1$ and $D_2$ are coprime and using the fact that all supersingular elliptic curves are isogenous over $\overline{\mathbb{F}}_{p}$, we compute   $\#\{p\leq x: E_{1, \mathfrak{p}}\sim E_{2, \mathfrak{p}}\}\sim \#\{p\leq x: \left(\frac{D_1}{p}\right)= \left(\frac{D_2}{p}\right)=-1\} \sim \frac{1}{4}\frac{x}{\log x}$.  However, this does not predict the behavior of the  $m$-isogenies. 
\end{rmk}

\begin{rmk}
The size of $\pi(m)$ is also closely  related to  the minimal isogeny degree of elliptic curves over finite fields. The study of the minimal isogeny degree problem has applications in  cryptography and algorithmic number theory \cite{LoBo2020}. Analogues of this problem in the context of global number fields and function fields have been explored by Masser-W\"{u}stholz \cite{MaWu1990} and Griffon-Pazuki \cite{GrPa2022}. In  Section \ref{sec: minimal isogeny degree}, we show that  the lower bound in Theorem \ref{thm:main} is not bad in the sense that, for any sufficiently small $\epsilon>0$, the lower bound of  $\#\pi(m)$ can not surpass $m^{\frac{5}{6}+\epsilon}$ for most $m$, unless there is a nontrivial upper bound of the minimal isogeny degree for a family of elliptic curves over finite fields. 
\end{rmk}

%\begin{rmk}
%  The condition on $D_i$ are technical, and can be removed with a bit more work. 
%\end{rmk}
\begin{rmk}
We provide a lower bound for the primes $p\in \pi(m)$ in terms of $m$.  We recall that for any $\pf$ of $H$, if  $E_{1, \pf}$ and $E_{2, \pf}$ are $m$-isogenous, then  $E_{1, \pf}$ and  $E_{2, \pf}$  are both supersingular and defined over $\F_{p^2}$. In fact, by carefully choosing models of $E_{1, \pf}$ and $E_{2, \pf}$, any isogeny between them is defined over $\F_{p^2}$   \cite[2.3.7]{GorenLove2024}. Therefore, for $p\neq  2, 3$,   considering the impact  of the automorphism group  $\text{Aut}(E_{i, \pf})$ for $i=1,2$,   we  may assume $E_{1, \pf}$ and $E_{2, \pf}$ are $m$-isogenous over $\F_{p^{24}}$. In particular, the kernel of each $m$-isogeny from $E_{1, \pf}$ to $E_{2, \pf}$ is contained in $E_{1, \pf}(\F_{p^{24}})$. 
So we get
$
m\mid \#E_{1, \pf}(\F_{p^{24}}).
$
Finally, by \cite[Theorem 4.1]{Wa1969}, %(also see \url{https://mathoverflow.net/questions/444846/why-all-supersingular-elliptic-curves-over-bar-mathbbf-p-are-isogenous}), 
 we obtain  
 \[
 m\mid \#E_{1, \pf}(\F_{p^{24}})=p^{24}-2p^{12}+1=(p^{12}-1)^2.
 \]
In particular, this implies that if $p\neq 2,3$ and $p\in \pi(m)$, then
 \begin{equation}\label{eq:prime-lowerbound}
     p\geq \begin{cases} 
      (m-1)^{\frac{1}{6}} & \text{ $m$ is a prime},\\
     m^{\frac{1}{12}} & \text{ $m$ is squarefree},\\
     m^{\frac{1}{24}} & \text{ otherwise}.
     \end{cases}
 \end{equation}
\end{rmk}

%%%Let $z_i \in Y := \SL_2(\Zb)\backslash \Hb$ be the CM points corresponding to $E_i$, and
Let $j(E_i)$ be the $j$-invariant of $E_i$, which is an algebraic integer, and
$\varphi_m \in \Zb[X, Y]$ the $m$-th modular polynomial. Then we can characterize the set $\pi(m)$ as
\begin{equation}
  \label{eq:pim2}
  \pi(m) = \{ p \text{ prime}: p \mid \Nm(\varphi_m(j(E_1), j(E_2))),\, p\nmid m\}.
\end{equation}
%\todo{EA: Is it in principle possible that $p$ divides $m$ and the norm? \\ TW: I think it might be true in general as well. Because it is known that  $\Phi_p(X, Y)\equiv (X^p-Y)(X-Y^p)\pmod p$. Hence, $p\mid \phi_p(j(E_1), j(E_2))$ implies $j(E_{1, p})= j(E_{2, p})^p $ or $j(E_{1, p})^p = j(E_{2, p})$. But this means there is a Frobenius map between $E_{1, p}$ and $E_{2, p}$. However,  The problem is that Frobenius endomorphism is not a separable isogeny and something strange will happen.  Moreover, most of the references I know exclude $p\mid m$: \url{https://www.math.mcgill.ca/goren/PAPERSpublic/RISAT.pdf}, \url{https://mathweb.ucsd.edu/~kedlaya/ants10/sutherland/paper.pdf}. EA: Thanks. Since for us these primes are un-important it seems reasonable to exclude them as well then.}
 % \todo[inline]{YL: find out whether $p\nmid m$ is needed or not.\\
 % TW: added the error term. }
The following result is a direct consequence of Theorem \ref{thm:main} and offers an $S$-unit generalization of the main result in \cite{Li21}. 
\begin{cor}
  \label{cor:S-unit}
  Suppose $D_1, D_2$ are fundamental discriminants and co-prime. For any finite set $S$ of primes, the algebraic integer $\varphi_m(j(E_1), j(E_2))$ is not an $S$-unit for all $m \gg_\epsilon( \# S)^{5+\epsilon}$. 
\end{cor}

We now briefly outline the proof of Theorem \ref{thm:main}.
The proof proceeds along a similar line of argument in \cite{Li21}, utilizing higher Green functions and theorems of Gross-Zagier and Gross-Kohnen-Zagier.
The main innovation is to control the valuation of $\Nm(\varphi_m(j(E_1), j(E_2)))$ at $p$ by relating it to Fourier coefficients of cusp forms with increasing level. This is possible as such valuations are Fourier coefficients of  incoherent Eisenstein series with full level, and can be approximated by Fourier coefficients of coherent Eisenstein series with higher levels, which are holomorphic Hilbert modular forms. 
 This is the content of Theorem \ref{thm:est}. 
 After this, we still need a bound of such Fourier coefficients in terms of the level.
 This is achieved in  Theorem \ref{theorem} using Deligne's bound and a bound of the Petersson norm of a cusp form in terms of its Fourier coefficients at the cusp $\infty$.
 Section \ref{sec:pf-thm} contains the proof of Theorem \ref{thm:main}.\\

 \noindent {\bf Acknowledgment}: E.\ Assing is supported by the Germany Excellence Strategy grant EXC-2047/1-390685813 and also partially funded by the Deutsche Forschungsgemeinschaft (DFG, German Research Foundation) – Project-ID 491392403 – TRR 358.
Y.\ Li is supported by the Deutsche Forschungsgemeinschaft (DFG) through the Collaborative Research Centre TRR 326 ``Geometry and Arithmetic of Uniformized Structures'' (project number 444845124)
 and the Heisenberg Program ``Arithmetic of real-analytic automorphic forms''(project number 539345613).
This work started during our visit at the MPIM Bonn, which we thank for the friendly atmosphere.

\section{Preliminary}

% \todo[inline, caption=notations.]{YL: notations to be defined\\
%   $M_k(N)$,   }

We follow \cite{Li21} to recall some preliminary notions.
\subsection{CM Points}
We continue with the notation from the introduction. 
For $i = 1, 2$, let $z_i = \frac{-B_i + \sqrt{D_i}}{2A_i}  \in \Hb \cap K_i$ %\todo{TW: Did we define $\Hb$ before?} 
be a representative of $E_i$, with $\Hb$ the upper half plane, such that $A_i \in \Nb$ and $\gcd(A_1, A_2) = \gcd(A_i, 4D_1D_2) = 1$. 
Let $\af_i := A_i(\Zb  + \Zb z_i) \subset \Oc_{D_i}$ be the corresponding $\Oc_{D_i}$-ideal with norm $A_i$.
Denote
%$F := \Qb(\sqrt{D})$ the real quadratic subfield of $K := K_1 K_2$, where $D := D_1D_2$, and
$'$ the Galois conjugation of $F/\Qb$.

%Let $W = K$ \todo{Maybe we want to view $K$ as a quadratic extension, and $W$ as the corresponding quadratic space? Also, $F$ is not defined in the Introduction.}
We view $K$ as an $F$-quadratic space with the quadratic form $Q_F(x) = \frac{x\bar{x}}{ A\sqrt{D}},A := A_1A_2$, and denote it by $W$. 
%%%\todo[color=yellow]{EA: $\bar{\cdot}$ is the Galois conjugate of $K/F$?\\YL: yes. It's also just complex conjugation since $K/F$ is a CM field. EA: Thanks.}
Then $\chi:= \chi_W = \chi_{K/F} = \otimes_{v \le \infty} \chi_v$ is the quadratic Hecke character corresponding to $K/F$. 
For $t \in F^\times$, 
%and $\pf$ a place of $F$,
denote
\begin{equation}
  \label{eq:Diff}
  \Diff(W, t) := \{\pf < \infty \text{ place of }F: W \otimes F_\pf \text{ does not represent }t \}. 
\end{equation}
When the absolute norm $\Nm(t)$ is positive, resp.\ negative, this set has odd, resp.\ even, cardinality.
Furthermore, it contains only places non-split in $K$.
For this particular quadratic space $W$, we have
\begin{equation}
  \label{eq:Diff-conj}
\Diff(W, -t') = \{\pf': \pf \in \Diff(W, t)\}.   
\end{equation}

%%%Let $z_i := \frac{D_i + \sqrt{D_i}}2$. 
The restriction of scalars of $W$ to $\Qb$
%\todo[color=yellow]{TW: restriction to $\mathbb{Q}$.\\ YL: changed.}
%We can identify $(V, Q)$ with the rational quadratic space
is isomorphic to $(M_2(\Qb), \det)$ via the map
%\todo[color=yellow]{TW: We only used the notation $(V, Q)$ twice. I think we can just use $(M_2(\mathbb{Q}), \det)$.\\YL: changed.}
% via the map
\begin{align*}
A\cdot   \sum_{i = 1}^4 &x_i e_i \in K \mapsto \pmat{x_3}{x_1}{x_4}{x_2} \in M_2(\Qb),\\
e_1 &:= 1,~
e_2 := -\overline{z_1},~
e_3 := z_2,~
e_4 := -\overline{z_1}z_2,
\end{align*}
which identifies $\Oc_K \cong L = M_2(\Zb)$.
%\todo[color=yellow]{TW: I believ this is a standartd fact, but I am confused about the compatibility betweeen norm and det. For example, why  $Q_F(A e_1) =\det((0,1; 0, 0))=0$?\\ YL: this requires a bit calculation. For example, $\tr Q_F(A e_1) = \tr (Ae_1) \overline{(Ae_1)}/(A\sqrt{D}) = \tr A/\sqrt{D} = 0$.}
% Let $\zm_i \in \Hb$ be a representative of $E_i$, and we have the CM cycle

Let $\Gamma = \SL_2(\Zb)$ and $Y := \Gamma\backslash \Hb$ the modular curve of level 1. 
On $Y^2$, 
%$X_V$, the Shimura variety associated to $\GSpin(V)$,
the points $z_i$ determine a CM cycle %%is given by
\begin{equation}
  \label{eq:bigCM1}
  \begin{split}
Z = Z(D_1, D_2) = 
\sum_{\sigma \in \Gal(H/K)} (z_1, z_2)^{\sigma} +
(-\overline{z_1}, z_2)^{\sigma} + 
(z_1, -\overline{z_2})^{\sigma} +
(-\overline{z_1}, -\overline{z_2})^{\sigma}.
  \end{split}
\end{equation}
%\todo[color=yellow]{TW: It is better to use $K$ rather than $K_1K_2$\\ YL: changed.}
This is the "big CM cycle" associated to $W$, which was defined and studied in \cite{BKY12}.
Since $D_1, D_2$ are co-prime, this cycle only depends on $D_1, D_2$.  %\todo{JX: Maybe we can cite Bruiner--Kudla--Yang to say that this is the ``Big CM cycle associated to $W/F$'' defined there.}

%%%%Since $D_1, D_2$ are co-prime, we have $\Gal(H/K) \cong \Gal(H_1/K_1) \times \Gal(H_2/K_0$. %%%, and the orbit above can be written as 

Denote  $\Wb = \otimes_{v \le \infty} \Wb_v$  the incoherent $\Ab_F$-quadratic space such that $\Wb_{v}$ is positive definite for $v \mid \infty$, and $\Wb_v \cong W_v := W(F_v)$ for all $v < \infty$.
Such incoherent $\Wb$ and the $F$-quadratic space $W$ above are called \textit{neighbors at $\infty$}, or simply neighbors.
In general, for each place $\pf$ of $F$,
%%%For a finite prime $\pf$ of $F$,
denote $W^{(\pf)}$ the $F$-quadratic space such that
\begin{equation}
  \label{eq:Wp}
  W^{(\pf)}_v
  \cong
  \begin{cases}
    \Wb_v & \text{ if } v \neq \pf     ,\\
    \Wb_\pf^- & \text{ if } v = \pf.
  \end{cases}
\end{equation}
Here $\Wb_\pf^-$ is the local $F_\pf$-quadratic space with the same dimension, quadratic character, but the opposite Hasse invariant as $\Wb_\pf$.
The space $W^{(\pf)}$ is the neighbor of $\Wb$ at $\pf$. 
We denote $\hat W := W(\Ab_{F, f})$.

\subsection{Incoherent Eisenstein Series}

We keep  the notation from the earlier section and
recall incoherent Eisenstein series associated to $W$ following \cite{BKY12}. 
%%%Let $\psi_F = \psi_\Qb \circ \tr_{F/\Qb}$ be the additive character, where $\psi_{\Qb, \infty}(x) = \ebf(-x)$, and
Let $\omega
%%%= \omega_{W, \psi}
= \otimes_v \omega_v$ be the Weil representation of $\SL_2(\Ab_F)$ on the space $\Sc(W(\Ab_F))$ of Schwartz functions on $W(\Ab_F)$, 
%\todo[color=yellow]{TW: I think it is better to mention $\Sc$ is the Schwartz space.\\YL: changed.}
and $I(s, \chi) = \mathrm{Ind}^{\SL_2}_B |\cdot|^s \chi$ the induced representation. 
Then we have a section map $\lambda_v:\Sc(W_v) \to I(0, \chi_v)$ via $\lambda_v(\phi)(g) := (\omega_v(g)\phi)(0)$.

For any standard $ \Phi(\cdot, s) \in I(s,\chi)$, 
%%%\todo{TW: $I(s, \chi)$ is the prinncipal series representation.}
%
we can associate the normalized Eisenstein series
\begin{equation}
  \label{eq:Eisenstein}
  E^*(g, s, \Phi) := \Lambda(s+1, \chi)
  \sum_{\gamma \in B(F)\backslash \SL_2(F)} \Phi(\gamma g, s), 
\end{equation}
with $\Lambda(s, \chi) = (\Delta_K/\Delta_F)^{s/2} \pi^{-s-1}\Gamma((s+1)/2)^2 L(s, \chi)$, and
$\Delta_K, \Delta_F$ the discriminants of $K, F$, respectively. 
For a standard Siegel section $\Phi = \lambda(\phi) \in I(s, \chi)$ with $\phi = \phi_f \otimes \phi^{(1, 1)}_\infty \in \Sc(\Wb)$ such that $\lambda_\infty(\phi^{(1, 1)}_{\infty})$ is right $K_\infty$-equivariant of weight $(1, 1)$ and normalized to be $\lambda_\infty(\phi_{\infty})(1) = 1$, the Eisenstein series $E^*(g, s, \Phi)$ vanishes at $s = 0$, and we call it \textit{incoherent}.
The function
\begin{equation}
  \label{eq:Ec}
  \Ec((\tau_1, \tau_2), \phi_f) := (v_1v_2)^{-1/2} E^{*,\prime}(g_{\tau_1, \tau_2}, 0, \Phi)
\end{equation}
%\todo{YL: $v_i = im(\tau_i)$}
is a real-analytic Hilbert Eisenstein series of weight $(1, 1)$, where $v_i = \mathrm{Im}(\tau_i), g_{\tau_1, \tau_2} = (g_{\tau_1}, g_{\tau_2}) \in \SL_2(F\otimes \Rb), g_\tau = \smat{1}{u}{}{1} \smat{\sqrt{v}}{}{}{1/\sqrt{v}} \in \SL_2(\Rb)$ for $\tau = u + iv \in \Hb$.
%\todo[color=yellow]{TW: We did not define $n(\cdot)$ and $m(\cdot)$. EA: It seems that this notation is not used again below. So maybe we can just give $g_{\tau}$ in matrix form?\\YL: changed.}
Similarly, for a section $\Phi = \lambda(\phi)$ with $\phi = \phi_f \otimes \phi^{(1, 1)}_\infty \in \Sc(W^{(\pf)}(\Ab_F))$, the Eisenstein series
\begin{equation}
  \label{eq:Ep}
  \Ep((\tau_1, \tau_2), \phi_f) := (v_1v_2)^{-1/2} E^{*}(g_{\tau_1, \tau_2}, 0, \Phi)
\end{equation}
is a holomorphic Hilbert Eisenstein series of weight $(1, 1)$. 
 %\todo{TW: Probably mention for $t\in F^\times$, }  
 For totally positive $t \in F$, its $t$-th Fourier coefficient is given by
\begin{equation}
  \label{eq:ctphi}
  c(t, \phi_f) = -4 \prod_{v < \infty}
   \frac{W^{*}_{t, v}(0, \phi_v)}{\gamma(W_{v})}. 
 \end{equation}
% when $t$ is totally positive or zero.
 Otherwise $c(t, \phi_f)$ is identically 0.
Here $W^*_{t, v}(s, \phi)
= |\Delta_K/\Delta_F|_v^{-\frac{s+1}2} L_v(s+1, \chi_v) W_{t, v}(s, \phi)
$ is the local Whittaker function and $\gamma(W_v)$ is the Weil index (see e.g.\ \cite{YYY21}).

%Suppose $\phi_f = \otimes_{v} \phi_v$ with $\phi_v \in S(\Wb_v)$.
On the other hand, for a totally positive $t \in F$, the $t$-th Fourier coefficient $a(t, \phi_f)$ of $\Ec((\tau_1, \tau_2), \phi_f)$
is given by
%Here the coefficient $a(t, \phi)$ for
\begin{equation}
  \label{eq:atphi}
  a(t, \phi_f) =
  \begin{cases}
    -4 \frac{W^{*,\prime}_{t, \pf}(0, \phi_{\pf})}{\gamma(W_{\pf})}
    \prod_{v \nmid \pf \infty} \frac{W^{*}_{t, v}(0, \phi_v)}{\gamma(W_{v})} & \text{if } \mathrm{Diff}(W, t) = \{\pf\},\\
    0 & \text{otherwise.}
  \end{cases}
\end{equation}
% defined by
% \begin{equation}
%   \label{eq:Wtv}
% W_{t, v}(s, \phi)  :=  \int_{F_v} (\omega(wn(b))\phi)(0) |a(wn(b))|^s_v \psi_v(-tb) db.
% \end{equation}
%and $ A = \Nm(\df \Delta_{K/F})$.
%%%Note that $\mathrm{Diff}(W, t)$ is odd and could only contain non-split primes in $F$.
%\todo[color=yellow]{TW: The formula seems to be in a wrong place. \\YL: changed.}

%Also, the definition of $W_{t, v}$ makes sense for any $F_v$-quadratic space $U$ and $\phi \in S(U)$. 
We now recall some formulas of $W^{}_{t, v}(0, \phi_v)$ from \cite[\textsection 5]{YYY21}. 
%%for some choices of $U$ and $\phi$. 
\begin{prop}
  \label{prop:YYY}
  Let $\pf$ be a prime of $F = \Qb(\sqrt{D})$ inert in $K$. 
  %Let $E/F$ be an unramified quadratic extension of local fields.
  For the $F_\pf$-quadratic space $(U_\delta, Q_\delta) = (K_\pf, \varpi^\delta \Nm_{K_\pf/F_\pf}/\sqrt{D})$ 
  %\todo{TW: Probably it is better to distinguish the $Nm$ here with the norm of an ideal. } 
  with $\varpi$ a uniformizer of $\Oc_{F_\pf}$ and $\delta \in \Nb$, if $\phi_r = \cha(\varpi^r \Oc_{K_\pf})$ is the characteristic function of the subset $\varpi^r \Oc_{K_\pf} \subset K_\pf$
%\todo[color=yellow]{TW: I think it is better to mention that $\cha(\cdot)$ is the characteristc function.\\YL: changed.}
  with $r \in \Nb$, then we have
  \begin{equation}
    \label{eq:Wtv-val}
%    |d_{E/F}|^{-1/2} %%% d_{E/F} is trivial for E/F ur. 
|\sqrt{D}|_\pf^{-1/2}    L(1, \chi_{\pf})    \frac{W_{t, \pf}(0, \phi_r)}{\gamma(U_\delta)} =
    \begin{cases}
      1 & \text{ if }o(t\sqrt{D}) \ge \delta + 2r \text{ and } 2 \mid(o(t\sqrt{D})-\delta),\\
      0 & \text{otherwise,}
    \end{cases}
  \end{equation}
where  $o(t):=\ord_{\varpi}(t)$ for $t\in F_\pf^\times$. 
%%%  \todo{TW: $o(t):=\ord_{\varpi}(t)$ for $t\in F$. \\YL: yes.}
  When $o(t\sqrt{D}) \ge \delta + 2r$ and $2 \nmid o(t\sqrt{D})-\delta$, we have 
  \begin{equation}
    \label{eq:Wtv-der}
%    |d_{E/F}|^{-1/2}
|\sqrt{D}|_\pf^{-1/2}    L(1, \chi_{\pf})    \frac{W'_{t, \pf}(0, \phi_r)}{\gamma(U_\delta)} =
    \lp \frac{o(t\sqrt{D}) - \delta + 1}2 -r - p^{1-\delta -2r} \frac{1 - p^{\delta + 2r}}{p^2 - 1}\rp \log p, 
  \end{equation}
  where $p = |\Oc_{F, \pf}/\varpi| = \Nm(\pf)$ is a rational prime.
  Otherwise, it is 0.
\end{prop}
%\todo[color=yellow]{TW: Since $\frak{p}$ is inert in $K$, I think the rationnal prime $p$ splits in $F$, so  $\Nm(\frak{p})=p$?\\YL: changed.}

\begin{proof}
  For $r = 0$, this follows directly from Proposition 5.7 in \cite{YYY21}. 
  Note that we need to use the first variant in \textsection 5.1 loc.\ cit.\ since $\psi_F'(x) := \psi_F(x/\sqrt{D})$ has trivial conductor.
  The case of $r \ge 1$ follows from using additionally the second variant in \textsection 5.1. 
\end{proof}

It was proved in \cite{Li21} that $a(t, \phi_{\af_1, \af_2}) \le 0$ for all $t$.
For later purpose, we define
\begin{equation}
  \label{eq:ap}
  \cp(t, \phi) :=
  -    \frac{a(t, \phi)}{\log \Nm(\pf)}
  =
    \frac{4}{\log \Nm (\pf)} \frac{W^{*,\prime}_{t, \pf}(0, \phi_{\pf})}{\gamma(W_{\pf})}
    \prod_{v \nmid \pf \infty} \frac{W^{*}_{t, v}(0, \phi_v)}{\gamma(W_{v})}   
  .
%%%  -    \frac{a(t, \phi)}{e_\pf\log \Nm(\pf)},
%   \begin{cases}
% -    \frac{a(t, \phi)}{2\log \Nm(\pf)}& \text{if } \mathrm{Diff}(W, t) = \{\pf\},\\% \text{ with } \Nm( \pf) = p,\\
%     0& \text{otherwise.}
%   \end{cases}
\end{equation}
%%where $e_\pf$ is the ramification index of $\pf$ over $\Qb$. 
Using \eqref{eq:Diff-conj}, we have
\begin{equation}
  \label{eq:cp-conj}
\cpp(t, \phi) = \cp(-t', \phi)  
\end{equation}
for all $\phi \in \Sc(\hat W)$.

\subsection{Higher Green Function}
Define a function $g_s$ on $\Hb^2$ by
\begin{equation}
  \label{eq:gs}
  g_s(z_1, z_2) :=
%  -2Q_{s-1} (\cosh \mathrm{d}(z_1, z_2)) =
  -2 Q_{s - 1} \lp 1 + \frac{|z_1 - z_2|^2}{2y_1y_2} \rp
\end{equation}
for $(z_1, z_2) \in \Hb^2$ with $Q_{s-1}(t)$ Legendre function of the second kind. 
%$\mathrm{d}(z_1, z_2)$ the hyperbolic distance between $z_1$ and $z_2$.
Averaging over the $\Gamma$-translates of $z_2$ gives  the function on $Y^2$
\begin{equation}
  \label{eq:Gs}
  G_s(z_1, z_2) := \sum_{\gamma \in \Gamma} g_s(z_1, \gamma z_2), \, \mathrm{Re}(s) > 1,
\end{equation}
which has logarithmic singularity along the diagonal.
%%%for $z_i \in Y$. % := \Gamma\backslash \Hb$. 
%on $Y^2$ symmetric in $z_1$ and $z_2$.

For $m \in \Nb$, we can let the $m$-th Hecke operator $T_m$ act on one of $z_1$ and $z_2$ to define
\begin{equation}
  \label{eq:Gsm}
  G^m_{s}(z_1, z_2) := \sum_{\gamma \in \Gamma \backslash \Gamma_m} G_s(z_1, \gamma z_2),
\end{equation}
where $\Gamma_m := \{\gamma \in M_2(\Zb): \det(\gamma) = m\}$. 
% It has logarithmic singularity along the divisor
% \begin{equation}
%   \label{eq:Tc}
%   \mathcal{T}_m := \{(z, \gamma \cdot z): z \in Y, \gamma \in \Gamma_m\} \subset Y^2.
% \end{equation}

For $f \in M_{2-2k}^!$ with $k > 1$ and Fourier expansion $ f(\tau) = \sum_{m \gg -\infty} c_f(m)q^{m} + O(1)$, we can define
\begin{equation}
  \label{eq:Gf}
  G_f(z_1, z_2) := \sum_{m \ge 1} c_f(-m) m^{k-1} G^m_k(z_1, z_2). 
\end{equation}
This is a regularized theta lift of $f$ \cite[Proposition~4.2]{Li18}. %\todo{TW: Proposition 4.2. EA: changed.}
When $k = 1$, this theta lift coincides with the Borcherds lift of $f$, which we also denote by $G_f$.

We now state the formula relating the value of $G_f$ at the CM cycle $Z$ in \eqref{eq:bigCM1}.
This result is originally due to Gross, Kohnen and Zagier \cite{GKZ87}. The following formulation is given in Theorem 3.5 in \cite{Li21}. 

\begin{thm}
\label{thm:GKZ1}
Let $f \in M_{2-2k}^!$ with $k \ge 1$ odd and vanishing constant term if $k = 1$. Suppose $D_1D_2$ is not a perfect square. Then
  \begin{equation}
    \label{eq:GKZ1}
    \begin{split}
      G_f(Z) &= \frac{2[H:K]}{\Lambda(0, \chi)} \sum_{m \ge 1} c_f(-m) a_{m}(\phi; k),\\
a_{m}(\phi; k) &:=      m^{k-1}\sum_{t \in F, t \gg 0, \tr(t) = m} P_{k-1}\lp \frac{t - t'}{m} \rp a_{}(t, \phi),
    \end{split}
  \end{equation}
%  \todo[color=yellow]{TW: Define $|Z|$\\YL: modified}
  where $a_{}(t, \phi)$ is the $t$-th Fourier coefficient of holomorphic part of the incoherent Eisenstein series ${E^{*,}}'((\tau_1, \tau_2), 0, \phi)$ with $\phi = \phi_{\af_1, \af_2} $ the characteristic function of $\overline{\af_1} \af_2 \otimes \hat{\Zb}$, $\Lambda(s, \chi)$ the completed $L$-function, and $P_{k-1}(x)$ the $(k-1)$-th Legendre polynomial.
\end{thm}

% We can then rewrite \eqref{eq:GKZ1} as
% $$
% G_k^m(Z(\Wb)) = \frac{|Z(\Wb)|}{\Lambda(0, \chi)}
% \sum_{p \text{ rational prime, non-split in }F} \log p
% \sum_{t \in S_m} P_{k-1}\lp \frac{t - t'}{m} \rp \cp_{}(t, \phi).
% $$
% %\todo{TW: Define $G_k^m(Z(\Wb))$ and $S_m$.\\ YL: simplify evaluating $G^m_k$ at the 0-cycle $Z(\Wb)$}

%%%\todo[inline]{YL: write down formula below for general $D_1, D_2$.}
Suppose that $D_1, D_2$ are fundamental and co-prime. Then $K/F$ is unramified, and  $\Gal(H/K) \cong \Gal(H_1/K_1) \times \Gal(H_2/K_2)$ and $\phi_{\af_1, \af_2} = \cha(\hat\Oc_K)$.
For $k \in \{1, 3, 5, 7\}$, we can take $f = q^{-m} + O(q) \in M^!_{2-2k}$ with any $m \in \Nb$ to deduce from the above theorem
% The formula above simplifies to
\begin{equation}
  \label{eq:main2}
  \sum_{\sigma_i \in \Gal(H_i/\Qb), i = 1, 2}
G_k^m(z_1^{\sigma_1} , z_2^{\sigma_2})
 = -{w_1w_2}
\sum_{\pf \text{ inert in }K} \log \Nm(\pf)
 \sum_{t \in \df^{-1}, t \gg 0, \tr(t) = m} P_{k-1}\lp \frac{t - t'}{m} \rp \cp(t, \phi)
\end{equation}
with $w_i := |\Oc_{D_i}^\times|/2$.
Here Proposition~\ref{prop:YYY} %\todo{TW: Proposition 2.2. EA: Changed.} 
gives us
%$\cp(t, \phi)$
%The Fourier coefficient $a(t, \phi)$ is
%%given by
$$
% a(t, \phi) = -2(1 + \ord_\pf(t\sqrt{d})) \rho(t\sqrt{d}\pf^{-1}) \log \Nm(\pf)
\cp(t, \phi) = 2(1 + \ord_\pf(t\sqrt{D})) \rho_{K/F}(t\sqrt{D}\pf^{-1}) 
$$
when $\mathrm{Diff}(W, t) = \{\pf\}$, and 0 otherwise, where $\rho_{K/F}$ is the ideal counting function.
%\todo{TW: I might miss the definition but did we define   $\rho_{K/F}(t\sqrt{D}\pf^{-1})$ before? }

To prove Theorem \ref{thm:GKZ1}, one uses the  regularized theta lift expression of $G_f(z_1, z_2)$, the Siegel-Weil formula, Stokes' Theorem and the Rankin-Cohen operator, which is defined by
\begin{equation}
  \label{eq:RC}
(\RC_{k-1} G  )(\tau) := (-2\pi i)^{1-k} \sum_{j = 0}^{k-1} \binom{k-1}{j}^2 \partial_{\tau_1}^j \partial_{\tau_2}^{k-1-j} G(\tau_1, \tau_2) \mid_{\tau_1 = \tau_2 = \tau}. 
\end{equation}
for any real-analytic function $G: \Hb^2 \to \Cb$. 
The details are contained in \textsection 3 of \cite{Li21}. 

\subsection{Estimate at $p$}

In this section, we will use Fourier coefficients of coherent  Hilbert Eisenstein series to approximate those of incoherent   Hilbert Eisenstein series.
%%%The Eisenstein series $E^*_{W^{(\pf)}}$
%((\tau_1, \tau_2), \phi^{(\pf)})$ associated to $\phi^{(\pf)} \in \Sc(W^{(\p%f)}(\hat F))$
%%is holomorphic of parallel weight 1. 

First, we describe the local quadratic space $\Wb_\pf^- \cong W^{(\pf)}_\pf$. 
\begin{lemma}
  \label{lemma:Wp-}
  If $\pf$ is unramified in $K$, then $(\Wb^-_\pf, Q_\pf^-) \cong (\Wb_\pf, \varpi \cdot Q_\pf)$, where $\varpi$ is a uniformizer of $F_\pf$ such that $(\varpi, D_1)_\pf = (\varpi, D_2)_\pf = -1$. 
%%%  If $\pf$ is ramified in $K$ ....\todo{YL: put in. This only happens when $\gcd(\Delta_1, \Delta_2) > 1$.}
\end{lemma}

\begin{rmk}
  From now on, we will use $ (\Wb_\pf, \varpi \cdot Q_\pf)$ as a model of $\Wb^-_\pf$ for local computations. 
\end{rmk}
\begin{proof}
  %% \todo{YL: put in. compute Hilbert symbol.}
  Denote $Q_\pf = \alpha \cdot\Nm_{K_\pf/F_\pf}$ for $\alpha \in F_\pf^\times$.
  % Since $D_1, D_2$ are co-prime, we suppose
  % without loss of generality that $\Nm( \pf) $ and $D_1$ are co-prime.
  Then $Q_\pf(a + b\sqrt{D_1}) = \alpha(a^2 - b^2D_1)$ and the Hasse invariant of $\Wb_\pf$ is
  $$
(\alpha, -\alpha D_1)_\pf = (\alpha, D_1)_\pf. 
$$
Since up to $F_\pf^\times$, $D_i$ is in $\Oc_{\pf}^\times$,  we can find uniformizer $\varpi \in \Oc_p$ such that $(\varpi, D_1)_\pf = -1$. Changing $\alpha$ to $\alpha \cdot \varpi$ negates the Hasse invariant, but preserves the determinant of the quadratic space. 
\end{proof}
For $\phi = \otimes_v \phi_v \in \Sc(\hat W)$, denote
\begin{equation}
  \label{eq:phip}
  \phi^{(\pf)} :=  \bigotimes_{v< \infty, v \neq \pf} \phi_v \in \bigotimes_{v < \infty, v \neq \pf}\Sc(\Wb_v). 
\end{equation}
Our next result relates its  $t$-th Fourier coefficient $c^{}(t, \tilde\phi_\pf \otimes \phi^{(\pf)})$
%%%%\todo{TW: Is this supposed to be  $c(t, \tilde\phi_\pf \otimes \phi^{(\pf)})$?\\ YL: fixed. }
of $\Ep((\tau_1, \tau_2), \tilde\phi_\pf \otimes \phi^{(\pf)})$
to 
% of the Eisenstein series for $W^{(\pf)}$.
% gives us an expression relating
the coefficient $\cp(t, \phi)$ for a suitable $\tilde\phi_\pf \in \Sc(\Wb_\pf^-) = \Sc(W^{(\pf)}_\pf)$.

\begin{thm}
  \label{thm:est}
  % Let $m \in \Nb$ be fixed.
  For an ideal $\af \subset \Oc_F$ and  $U = W$ or $W^{(\pf)}$,
  define $\phi_U(\af) := \cha(\af \hat \Oc_K) \in \Sc(\hat U)$.
  If $\af = \Oc_F$, we omit it from the notation. 
%   For any inert prime $\pf$ in $F$ and  $r \ge 0$,
% denote $\phi_r(\pf) = \phi_{\pf, r} \otimes \phi^{(\pf)}  \in \Sc(\widehat{ W^{(\pf)}})$ with
%   $\phi_{\pf, r} := -\cha(\varpi^r \Oc_{K_\pf})
%   \in \Sc(W_\pf^-)$ for a uniformizer
%   $\varpi \in \Oc_{F_\pf}$. 
  Then
  \begin{equation}
    \label{eq:coh-match}
    \sum_{r = 0}^{\infty}    c(t, \phi_{W^{(\pf)}}(\pf^r))
=    \sum_{r = 0}^{      \lfloor o_\pf(t)/2 \rfloor}    c(t, \phi_{W^{(\pf)}}(\pf^r))
  = \cp(t, \phi_W)
  \end{equation}
%  \todo[color=yellow]{TW: We only introduced the nnotation  $\phi_r$ but not $\phi_{W^{(\pf)}}(\pf^r)$. \\ YL:it's $\phi_U(\af)$ with $U = W^{(p)}, \af = \pf^r$.}
  for all totally positive $t \in F$ and inert prime $\pf$ of $F$.
  %with $\mathrm{Diff}(W, t) = \{\pf\}$. 
%%%  Furthermore, the level of $\tilde\phi_{\pf, r}$ is $p^{r}$.
%%%  \todo{YL: be ore precise.}
\end{thm}
\begin{rmk}
   Using Proposition~\ref{prop:YYY}, %\todo{TW: Proposition 2.1. EA: Changed.}
   we can write 
 \begin{equation}
   \label{eq:ctr}
   c(t, \phi_{W^{(\pf)}}(\pf^r))
   = \rho_{K/F}(t\sqrt{D}/\pf^{2r+1}),
 \end{equation}
 which is valid for all totally positive $t$. %\todo[color=yellow]{TW: I think it is bettr to say for all sufficiently large $t$? \\ YL: modified.}
\end{rmk}
\begin{proof}
  Since $W^{(\pf)}_v \cong W_v$ for all $v < \infty$ different from $\pf$, a prime $\lf \neq \pf$ is in $\Diff(W^{(\pf)}, t)$ if and only if it is in $\Diff(W, t)$.
  If such prime exists, then both sides of \eqref{eq:coh-match} vanishes. 

  So it is enough to check \eqref{eq:coh-match} for those $t$ with $\Diff(W, t) = \{\pf\}$, where it    is equivalent to
  %there exist $\tilde\phi_{r} \in \Sc(\Wb^-_\pf)$ such that 
  \begin{equation}
    \label{eq:coh-match2}
|\sqrt{D}|_\pf^{-1/2}    L(1, \chi_{\pf})    \sum_{r = 0}^{\infty} \frac{W_t(0, \cha(\widehat{\pf^r\Oc_K}))}{\gamma(W^-_\pf)}
%=    L(1, \chi_{E/F})    \sum_{r = 0}^{\lfloor o_\pf(t)/2 \rfloor} \frac{W_t(0, \tilde\phi_{\pf, r})}{\gamma(W^-_\pf)}
    =
|\sqrt{D}|_\pf^{-1/2}
    \frac{    L(1, \chi_{\pf})}{\log \Nm(\pf)} \frac{W'_t(0, \cha(\widehat{\Oc_K}))}{\gamma(W_\pf)} 
  \end{equation}
  for all $t \in F_\pf$
   by \eqref{eq:ctphi}, \eqref{eq:atphi}, and \eqref{eq:ap}. 
%   When $F_\pf = \Qb_p$, $\phi = \cha(\Oc_{K_\pf}) \in \Sc(\Wb_\pf)$, we have $Q_\pf = \Nm$ and
   Using Proposition \ref{prop:YYY}, we obtain that
   \begin{align*}
     |\sqrt{D}|_\pf^{-1/2}    L(1, \chi_{\pf})
     &\sum_{r = 0}^{\infty} \frac{W_t(0, \cha(\widehat{\pf^r\Oc_K}))}{\gamma(W^-_\pf)}
% =      L(1, \chi_{E/F})
%      \sum_{r = 0}^{\lfloor o_\pf(t)/2 \rfloor} \frac{W_t(0, \tilde\phi_{\pf, r})}{\gamma(W^-_\pf)}
     =
           \sum_{r = 0}^{\infty} \delta_{2\mid o(t)-1  \ge 2r}
       =        \sum_{r = 0}^{\lfloor o_\pf(t)/2 \rfloor} \delta_{2\mid o(t)-1  \ge 2r}\\
       &= \delta_{2 \nmid o(t)} \frac{o(t)+1}2
         =     
|\sqrt{D}|_\pf^{-1/2}
         \frac{    L(1, \chi_{E/F})}{\log \Nm(\pf)} \frac{W'_t(0, \cha(\widehat{\Oc_K}))}{\gamma(W_\pf)}. 
   \end{align*}
   % both sides are $     -\frac{o(t) + 1}2 \delta_{2 \nmid o(t) > 0}$
%    when $\gcd(\Nm(\pf), D) = 1$.
% Similarly    when $\gcd(\Nm(\pf), D) > 1$, we have
%   %  $2\nmid o(t) > 0$, and 0 otherwise.
%   % In other words, we can write
%   $$
% ....       = -\frac{\delta_{2 \mid o(t)}}2 \lp\frac{o(t)}2 - \frac{p}{p-1}\rp
%        =     -\frac{    L(1, \chi_{E/F})}{\log \Nm(\pf)} \frac{W'_t(0, \phi_\pf)}{\gamma(W_\pf)}. 
% %   -\frac{    L(1, \chi_{E/F})}{2\log p} \frac{W'_t(0, \phi)}{\gamma(\Wb_\pf)} 
% % = -\frac{\delta_{2 \nmid o(t)}}2 \sum_{r = 0}^{\lfloor o(t)/2 \rfloor} 1.
% $$
% By Lemma \ref{lemma:Wp-}, we have $(\Wb_\pf^-, Q) = (K_\pf, p \cdot \Nm)$.
% If we take
% then Proposition \ref{prop:YYY} shows that the LHS of \eqref{eq:coh-match2} agrees with the RHS above. 
This finishes the proof.
 \end{proof}

% Then we can rewrite \eqref{eq:main2} as
% $$
%   \sum_{\sigma_i \in \Gal(H_i/\Qb), i = 1, 2}
% G_k^m(z_1^{\sigma_1} , z_2^{\sigma_2})
%  = \frac{w_1w_2}2
% m^{k-1}\sum_{t = \frac{a + m\sqrt{d}}{2\sqrt{d}} \in \df^{-1}} P_{k-1}\lp \frac{t - t'}{m} \rp a_{}(t, \phi).
% $$

 Putting this Theorem into \eqref{eq:main2} gives us the following result.
 \begin{cor}
   \label{cor:1}
  Let $z_1, z_2$ be CM points with co-prime, fundamental discriminants $D_1, D_2 < 0$, and $w_i$ same as in \eqref{eq:main2}. For any $m \in \Nb$ and $k = 1, 3, 5, 7$, we have
  \begin{equation}
    \label{eq:Gm-est}
    \begin{split}
m^{k-1}            \sum_{\sigma_i \in \Gal(H_i/\Qb), i = 1, 2}
G_k^m(z_1^{\sigma_1} , z_2^{\sigma_2})
& = -{w_1w_2}
  \sum_{p \in \Nb \text{ prime}} \log p
\sum_{r = 0}^\infty c_{k}(m; p^{2r+1}) ,
    \end{split}
  \end{equation}
  where
  $
  c_k(m; p^{2r+1}) 
    % =\sum_{t \in \df^{-1}, t \gg 0, \tr(t) = m} P_{k-1}\lp \frac{t - t'}{m} \rp
    %     c(t,  \phi_{W^{(\pf)}}(\pf^r))
    $
    is the $m$-th Fourier coefficient of the elliptic modular form
    \begin{equation}
      \label{eq:gkr}
      g_{k, p, r} :=
      \sum_{\pf \mid p,~ \pf \text{ non-split in }K}
%%\frac{\log \Nm(\pf)}{\log p}
      \RC_{k-1} \Ep((\tau_1, \tau_2), \phi_{W^{(\pf)}}(\pf^r)) \in M_{2k}(p^{1 + 2r}),
    \end{equation}
%%where $p:= \nm(\pf) \in \Nb$ is the rational prime below $\pf$. 
%    where $p \in \Nb$ is the rational prime below $\pf$. 
which only depends on $D_i, k, p, r$. 
    Furthermore, 
    $g_{k, p, r}$ is a cusp form when $k > 1$ and identically 0 when $p \nmid \Nm(\varphi_m(j(z_1), j(z_2)))$. 
  \end{cor}
%\todo[color=yellow]{TW: Do we need to distinguish the cas $m=p$?\\YL: I don't think so.}
  \begin{rmk}
    \label{rmk:r-bound}
    By \eqref{eq:ctr}, the $t$-th Fourier coefficient of
    $ \Ep((\tau_1, \tau_2), \phi_{W^{(\pf)}}(\pf^r))$ is non-zero only if $p^{2r+1} = \Nm(\pf^{2r+1}) \mid \Nm(t\sqrt{D})$.

So    for a fixed $m$, the summation over $r$ is bounded above by $\log_p (m\sqrt{D}/2)$.
  \end{rmk}
  
  \begin{proof}
    Since there exists a prime $\pf \mid p$ of $F$ that is inert in $K$ only if $p$ is non-inert in $F$, we have $\Nm(\pf) = p$ and
 only need to verify that
    %the $m$-th Fourier coefficient of $g_{k, \pf, r}$ is given by
    \begin{equation}
      \label{eq:ckr}
      c_k(m; p^{2r+1}) =      m^{k-1}
      \sum_{\pf \mid p,~ \pf \text{ non-split in }K}
%\frac{\log \Nm(\pf)}{\log p}
      \sum_{t \in \df^{-1}, t \gg 0, \tr(t) = m} P_{k-1}\lp \frac{t - t'}{m} \rp        c(t,  \phi_{W^{(\pf)}}(\pf^r)).
          \end{equation}
          This follows directly from applying $\RC_{k-1}$ to the Fourier expansion of $\Ep$ (See Example 2.1 in \cite{Li23}).

          When $k > 1$, the operator $\RC_{k-1}$ will annihilate the constant term at all cusps, since it commutes with the slash action. So $g_{k, p, r}$ is a cusp form in such cases.
          If $p \nmid \Nm(\varphi_m(j(z_1), j(z_2)))$, then $c_1(m; p^{2r+1}) = 0$ for all $ r \ge 0$. By the non-negativity of $c(t, \phi_{W^{(\pf)}}(\pf^r))$, we  have
          $$
c(t, \phi_{W^{(\pf)}}(\pf^r)) = 0
$$
for all totally positive $t \in \df^{-1}$ with $\tr(t) = m$ and $r \ge 0$. This means $c_k(m; p^{2r+1}) = 0$ for all $k$ and $r$. 
  \end{proof}

\section{A uniform bound for the Petersson norm}
%\todo[color=yellow]{
  %EA: Should we work with weight $2k$ as in the sections above or leave weight $k$?\\YL: maybe use a different letter like $\kappa$ for $k$. Also ``of level $\Gamma_0(N)$'', weight is necessarily even.
  %YL: One can add a remark later if the result holds for $\Gamma_1(N)$ and half-integral weight. }
%\todo[inline]{EA: I have tried to tidy this up a bit. Feel free to remove unnecessary details or remarks. Also let me know if something is not clear.\\YL: I see. Perhaps we can have a remark here that trivial bound suffices for us, and what the expected sup-norm is, and how it would improve the exponent in Theorem \ref{thm:main}.\\ EA: Improvements on the sup-norm are only expected for newforms. For a general element in $S_k(N)$ our bound is sharp (in the level).}
Let $f\in S_{\kappa}(N)$ be a cusp form for $\Gamma_0(N)$ (i.e. of level $N$) and weight $\kappa\geq 2$. We put%\todo{EA: Should I rename the dimension from $D$ to $d$? So we avoid confusion with $D=D_1D_2$ in the other section?\\TW: I think that would be better.}
\begin{equation}
	d=\dim S_{\kappa}(N) \ll \kappa N\log\log(N), \label{eq:dimension}
\end{equation}
where the upper bound is taken from \cite[Theorem~6]{martin2005}.
%%\todo{A slightly better bound could be $D \ll \kappa N \log N$? EA: Why not. We have the epsilons anyway so it does not change much.}
The Fourier expansion of $f$ at the cusp $\infty$ is written as
\begin{equation}
	f(z) = \sum_{n=1}^{\infty} a_f(n)e(nz). \nonumber
\end{equation}
Finally, the Petersson norm of $f$ is defined by %\todo[inline]{JX: Since we are not assuming $f$ is a newform, the definition of the Petersson norm might still depend on $N$. Do we need a subscript $\Vert f\Vert_N$? \\EA: I think because we never use another inner product we don't need it. (Originally I had a basis with old and newforms and the inner products got a bit confusing so that I introduced the subscript.) }
%\todo{YL: I think $N$ in the subscript is unnecessary, since norm is normalized by volume of $\mathcal{F}_N$. EA: Removed.}
\begin{equation}
	\Vert f\Vert^2 = \frac{1}{[\textrm{SL}_2(\mathbb{Z})\colon \Gamma_0(N)]} \int_{\mathcal{F}_N} \vert  f(x+iy)\vert^2 y^{{\kappa}-2} dxdy. \nonumber
\end{equation}
Here $\mathcal{F}_N$ is a fundamental domain for $\Gamma_0(N)$, which we assume to be contained in the strip $$\mathcal{P}=\{ x+iy\in \mathbb{H}\colon 0\leq  x\leq 1\}.$$

We recall the bound
\begin{equation}
	\vert a_f(n)\vert \ll_{\epsilon} d^{\frac{1}{2}}N^{\frac{1}{2}+\epsilon}n^{\frac{\kappa-1}{2}+\epsilon}\Vert f\Vert \label{eq:schulze-pillot}
\end{equation}
from \cite[Theorem 12]{SP}. This bound holds for general $f$. It is a natural question to ask for suitable estimate for the the Petersson norm $\Vert f\Vert$ in terms of the Fourier coefficients of $f$ (at infinity). For small levels this has been considered in \cite{Jenkins-Rouse2011, Jenkins_Pratt2015, Choi-Im-2018}. In Theorem~\ref{theorem} below we give a general answer.

% \todo[inline]{YL: also mention works for small level (Jenkins-Rouse 2011, Jenkins-Pratt 2015, Choi-Im JNT 2018) EA: I have added references to Jenkins et al above this remark. I had a look at Choi-Lim, but I do not understand what they are doing. Feel free to add the reference anyway.\\
% YL: I meant the paper ``Bounds for the coeﬃcients of cusp form for $\Gamma_0(3)$'' by SoYoung Choi and Bo-Hae Im in JNT (2018). They gave the bound for level 3, Fricke eigenforms. \\
% EA: I see. Sorry I must have mixed this up.}

\begin{rmk}
A similar problem occurs in the classical literature concerning quadratic forms. Here one naturally encounters the problem of uniformly estimating the Petersson norm of the cuspidal part $f$ of certain theta functions. See \cite[Section~3]{waibel2022} and the references within for further discussion.

The standard approach is to cover the fundamental domain $\mathcal{F}_N$ by a suitable collection of Siegel sets. Executing the integrals over these sets yields an upper bound for the Petersson norm of $f$ in terms of sums of Fourier coefficients. Typically this estimate involves the Fourier coefficients of $f$ at all cusps. In the case where $f$ arises from a theta function, one has sufficient control on all Fourier coefficients to obtain the desired estimate for the Petersson norm of $f$.

However, for general $f\in S_{\kappa}(N)$, the Fourier coefficients at cusps different from $\infty$ can be hard to access. To circumvent this issue we modify the approach described above by using different Siegel sets.
\end{rmk}

We introduce the sup-norm
\begin{equation}
	\Vert f\Vert_{\infty} = \sup_{z\in \mathbb{H}} \Im(z)^{\frac{\kappa}{2}} \vert f(z)\vert.\nonumber
\end{equation}
For a general form $f\in S_{\kappa}(N)$ one has the bound %\todo{TW: Do you mean ``lower bound"? EA: No it is supposed to mean local, because it does not require any understanding of the 'global' geometry of $\Gamma_0(N)\backslash \Hb$. But maybe its better to remove the adjective.\\TW: Thanks!)}
\begin{equation}
	\Vert f\Vert_{\infty}^2 \ll [\textrm{SL}_2(\mathbb{Z})\colon \Gamma_0(N)]\cdot \kappa^{\frac{3}{2}} \cdot \Vert f\Vert^2.  \label{eq:sup_norm}
\end{equation}
This is a direct consequence of general bounds for the Bergman kernel as given in \cite{kramer} for example.

\begin{remark}
The value of the estimates \eqref{eq:schulze-pillot} and \eqref{eq:sup_norm} lies in their generality. Indeed, they do not require any assumption on the form $f\in S_{\kappa}(N)$. The price paid for this is that they are not sharp in general. 
\end{remark}

Let $Y>0$ be a small parameter.  We define the Siegel set
\begin{equation}
	\mathcal{P}_Y=\{ x+iy\in \mathbb{H}\colon 0\leq  x\leq 1 \text{ and }y>Y\}. \nonumber
\end{equation}
A standard computation shows that
%%%\todo{YL: calculation below is standard, can omit details. EA: removed.}
\begin{equation}
	I_f(Y) %&= \int_{\mathcal{P}_Y}\vert f(x+iy)\vert^2y^{k-2}dxdy \nonumber \\
	%&= \sum_{n,m\in \mathbb{N}} a_f(n)\overline{a_f(m)} \int_Y^{\infty}e^{-2\pi(n+m)y}y^{k-2}dy\int_{0}^{1}e((n-m)x)dx \nonumber\\
	= \sum_{n=1}^{\infty}\vert a_f(n)\vert^2\int_Y^{\infty}e^{-4\pi ny}y^{\kappa-2}dy. \nonumber\\
	%= \sum_{n=1}^{\infty}\vert a_f(n)\vert^2 \cdot (4\pi n)^{1-k}\Gamma(k-1,4\pi nY)\nonumber
\end{equation}
%where $\Gamma(k-1,4\pi nY)$ denotes the incomplete Gamma function. By partial integration we see that it evaluates to
%\begin{equation}
%	\Gamma(k-1,x) = (k-2)!e^{-x}\sum_{r=0}^{k-2}\frac{x^r}{r!}.\nonumber
%\end{equation}
%In particular, we can estimate
%\begin{equation}
%	\Gamma(k-1,4\pi nY) \leq \begin{cases}
%		(k-1)! &\text{ if }n\leq \frac{1}{4\pi Y},\\
%		(k-1)\cdot e^{-4\pi nY}(4\pi nY)^{k-2} &\text{ if }n>\frac{1}{4\pi Y}.\nonumber
%	\end{cases} 
%\end{equation}
%Thus, we expect to be able to truncate the $n$-sum roughly at $n = Y^{-1}$, where the exponential decay kicks in.

The remaining $y$-integral decays exponentially for sufficiently large $n$. This allows us to truncate the $n$-sum as follows. 

\begin{lemma}
Let $f\in S_{\kappa}(N)$ arbitrary. For $\epsilon,\delta>0$ there is a (small) constant $C=C(\epsilon,\delta,\kappa)>0$ so that for $Y\leq C$ we have
\begin{equation}
	I_f(Y) \ll_{\kappa,\epsilon} S_f(Y^{-1-\delta}) +  YN^{2+\epsilon}\Vert f\Vert^2,\label{eq:IfY}
\end{equation}
where $$S_f(X)=\sum_{n\leq X} \vert a_f(n)\vert^2n^{1-\kappa}.$$
%\todo[inline]{YL: I think one can take $S_f(X)=\sum_{n\leq X} \vert a_f(n)\vert^2 n^{1-k}$ in the estimate for $\sum_{n \le K}$. EA: Do you mean removing the sup?\\ YL: I mean we can simply use $\sum_{n\leq X} \vert a_f(n)\vert^2 n^{1-k}$ instead of $S_f(X)$ in the statement. This is stronger, and gives better results later.\\ EA: Right. Changed it.\\ TW: The $\ll$ also depends on $\delta$?\\ EA: I don't think so.}
\end{lemma}
\begin{proof}

For a large parameter $K>0$ we use \eqref{eq:schulze-pillot} to estimate
\begin{align}
	&\sum_{n>K}\vert a_f(n)\vert^2 \cdot \int_Y^{\infty}e^{-4\pi ny}y^{\kappa-2}dy \nonumber\\
    &\qquad\ll_{\epsilon} dN^{1+\epsilon}\Vert f\Vert^2 \sum_{n>K} n^{\epsilon}\int_{Y}^{\infty}e^{-4\pi ny}(ny)^{\kappa-1}\frac{dy}{y} \nonumber\\
	&\qquad\ll_{\epsilon} dN^{1+\epsilon}\Vert f\Vert^2 \sum_{n>K} n^{\epsilon}\int_{nY}^{\infty}e^{-4\pi y}y^{\kappa-2}dy \nonumber\\
	&\qquad\ll_{\epsilon}  dN^{1+\epsilon}\Vert f\Vert^2 \int_{KY}^{\infty}e^{-4\pi y}y^{\kappa-2}\sum_{K\leq n<\frac{y}{Y}}n^{\epsilon}dy \nonumber \\
	&\qquad\ll_{\epsilon}  \frac{dN^{1+\epsilon}}{Y^{1+\epsilon}}\Vert f\Vert^2 \int_{KY}^{\infty}e^{-4\pi y}y^{\kappa-1+\epsilon}dy \nonumber\\
	&\qquad \ll_{\epsilon}N^{2+\epsilon}\kappa^{1+\epsilon}\Vert f\Vert^2\cdot \frac{e^{-(4\pi-1)KY}}{Y^{1+\epsilon}}\cdot \int_{KY}^{\infty}e^{-y}y^{\kappa-1+\epsilon}dy.\nonumber
\end{align}
If we assume that $YK>1$, which is a reasonable assumption, then we can estimate the remaining integral using repeated partial integration. We end up with
\begin{equation}
	\sum_{n>K}\vert a_f(n)\vert^2 \cdot \int_Y^{\infty}e^{-4\pi ny}y^{\kappa-2}dy \ll_{\epsilon}N^{2+\epsilon}\kappa^{1+\epsilon}\cdot \frac{e^{-4\pi KY}(KY)^{\kappa-1+\epsilon}}{Y^{1+\epsilon}}\cdot \Vert f\Vert^2.\nonumber
\end{equation}

On the other hand we compute
\begin{equation}
	\sum_{n\leq K}\vert a_f(n)\vert^2 \cdot \int_Y^{\infty}e^{-4\pi ny}y^{\kappa-2}dy   \leq\int_{Y}^{\infty}y^{\kappa-2}e^{-4\pi y}dy \sum_{n\leq K} \vert a_f(n)\vert^2n^{1-\kappa} \ll \frac{(\kappa-2)!}{(4\pi)^{\kappa-1}}\cdot S_f(K).\nonumber
\end{equation}
For fixed $\kappa$ we choose $K=Y^{-1-\delta}$ and obtain
\begin{equation}
 I_f(Y) \ll_{\kappa,\epsilon} S_f(Y^{-1-\delta}) + N^{2+\epsilon}\cdot e^{-4\pi Y^{-\delta}}Y^{-\delta(\kappa-1+\epsilon)-1-\epsilon}\cdot \Vert f\Vert^2\nonumber
\end{equation}
Note that for $x\gg_{\epsilon,\delta,\kappa} 1$ we have
\begin{equation}
    e^{-4\pi x}x^{\kappa-1+\epsilon} \leq x^{-\frac{2+\epsilon}{\delta}}. \label{some_equation}
\end{equation}
We take $C=C(\epsilon,\delta,\kappa)>0$ so that for $Y\leq C$ the inequality \eqref{some_equation} holds for $x=Y^{-\delta}$. The claimed statement follows directly.
\end{proof}

\begin{rmk}
In principle one can make the estimate \eqref{eq:IfY} explicit in $\kappa$. Since this is not required for our application, we  do not pursue this here.
\end{rmk}

%We now define
%\todo{YL: maybe $	\mathcal{F}_{N,Y} = \mathcal{F}_N\cap \mathcal{P}^c_Y$. Also, this notation is not used later.}
%\begin{equation}
%	\mathcal{F}_{N,Y} = \mathcal{F}_N\setminus \mathcal{P}_Y. \nonumber
%\end{equation}
Next we compute the hyperbolic volume of $\mathcal{F}_N\cap \mathcal{P}^c_Y$. Recall that $Y>0$ is a small parameter, so that this volume should be tiny.

The set of cusps of $\Gamma_0(N)$ can be parametrized by
% \todo{YL: maybe ``...the set of cusps...'' better. EA: Changed.\\ YL: sorry, I meant ``...the set of cusps of $\Gamma_0(N)$ can be...'' would be better. Also, the previous def of $\mathfrak{C}_N$ is okay, with $1/N$ instead of $\infty$. This way the width formula below works for all cusps in $\mathfrak{C}_N$. Also, the condition $\neq \infty$ can be changed to $\neq 1/N$. EA: Changed.}
\begin{equation}
	\mathfrak{C}_N = \{ \frac{u}{v} \colon v\mid N,\,  (u,v)=1 \text{ and } u \text{ mod }(v,\frac{N}{v}) \}.\nonumber
\end{equation} 
Note that $1/N$ is equivalent to the cusp $\infty$, which plays a special role for us. We can choose the fundamental domain $\mathcal{F}_N$, so that it has cuspidal vertices at the representatives in $\mathfrak{C}_N\setminus \{\frac{1}{N}\}$ and $\infty$.

The width of a cusp $\mathfrak{a}=\frac{u}{v}\in \mathfrak{C}_N$ is given by
%%%\todo{YL: this is true for all $\mathfrak{a} \in \mathfrak{C}_N$. EA: Right.}
\begin{equation}
	w_{\mathfrak{a}} = \frac{N}{(N,v^2)}.\nonumber
\end{equation}
Furthermore, we choose a matrix
\begin{equation}
	\sigma_{\mathfrak{a}} =  \left(\begin{matrix} u & \ast \\ v & \ast \end{matrix}  \right)\in \textrm{SL}_2(\mathbb{Z}). \nonumber
\end{equation}
In particular, $\sigma_{\mathfrak{a}}\infty = \mathfrak{a}$ and 	
\begin{equation}
	\sigma_{\mathfrak{a}}^{-1}\Gamma_{\mathfrak{a}}\sigma_{\mathfrak{a}} = \left\{ \pm \left(\begin{matrix} 1 & w_{\mathfrak{a}}m\\0&1\end{matrix}\right)\colon m\in \mathbb{Z}\right\}, \nonumber
\end{equation}
where $\Gamma_{\mathfrak{a}}\subseteq \Gamma_0(N)$ is the stabilizer of the cusp in question.
%%%\todo{YL: $\Gamma_{\mathfrak{a}} \subset \Gamma_0(N)$ is the stabilizer of the cusp in question. EA: Inclusion added.}

Let $K$ be a parameter, which we imagine large. We define the horocycle
\begin{equation}
	L(K) = \{ x+iK\colon x\in \mathbb{R}\}. \nonumber
\end{equation}
Given a matrix 
\begin{equation}
	\sigma = \left(\begin{matrix} u & \ast \\ v & \ast\end{matrix}\right)\in \textrm{SL}_2(\mathbb{Z}), \nonumber
\end{equation}
we observe that $\sigma L(K)$ is mapped to the circle tangent to $u/v$ of diameter $(v^2K)^{-1}$. In particular, $\sigma$ maps the set $\mathcal{P}_K$ to the interior of this circle. It is easy compute that
%%%%\todo{YL: do you mean $	\textrm{Vol}(\sigma\mathcal{P}_K) = \textrm{Vol}(\mathcal{P}_K) = \frac{1}{K}$? EA: Fixed.}
\begin{equation}
	\textrm{Vol}(\sigma\mathcal{P}_K) = \textrm{Vol}(\mathcal{P}_K) = \frac{1}{K}. \nonumber
\end{equation}
Using these observations we can estimate
\begin{align}
	\textrm{Vol}(\mathcal{F}_N\cap \mathcal{P}^c_Y) &\leq \sum_{\frac{1}{N}\neq \mathfrak{a}=u/v\in \mathfrak{C}_N}w_{\mathfrak{a}}\textrm{Vol}(\sigma_{\mathfrak{a}}\overline{\mathcal{P}_{(2v^2Y)^{-1}}}) \nonumber\\
	&= 2Y \sum_{\substack{v\mid N,\\ v\neq N}} \frac{Nv^2}{(v^2,N)}\cdot\varphi((v,\frac{N}{v})).\nonumber 
\end{align}
The remaining sum is multiplicative in $N$ and is easily evaluated on prime powers. Indeed, if we call
\begin{equation}
	\psi(N)=\sum_{\substack{v\mid N}} \frac{Nv^2}{(v^2,N)}\cdot\varphi((v,\frac{N}{v})),\nonumber 
\end{equation}
then we compute
\begin{equation}
	\psi(N) = N^2\cdot \prod_{p\mid N} (1+\frac{1}{p}) \ll_{\epsilon} N^{2+\epsilon}.\nonumber
\end{equation}
%\todo[inline]{TW: Could you please double check the computation? I got $\psi(p^e)=p^{2e}+p^{2e-1}$. I also checked it for 2 and  $2^2$. In any case, the bound $N^{2+\epsilon}$ is correct.\\ EA: Thanks for spotting this. I have double checked the ccomputations and you are correct. I don't know how I got the additional terms.}
We conclude that
\begin{equation}
	\textrm{Vol}(\mathcal{F}_N\cap \mathcal{P}^c_Y) \ll_{\epsilon} N^{2+\epsilon}Y.\label{eq:vol_est_cusps}
\end{equation}

We are now  ready to prove the main result of this section.

\begin{thm}\label{theorem}
Let $f\in S_{\kappa}(N)$ be an arbitrary non-trivial (i.e. $f\neq 0$) form. For a fixed $\epsilon>0$, there are constants $c_1=c_1(\epsilon,\kappa)$ and $c_2(\epsilon,\kappa)$ such that
\begin{equation}
	\Vert f\Vert^2 \leq \frac{c_1}{[\mathrm{SL}_2(\mathbb{Z})\colon \Gamma_0(N)]}\cdot\sum_{n\leq c_2N^{2+\epsilon}}\vert a_f(n)\vert^2n^{1-\kappa}. \nonumber
\end{equation}
Therefore the $m$-th Fourier coefficient $a_f(m)$ is bounded by%\todo[color=yellow]{Is there maybe a factor $\sqrt{N}$, missing in the estimate for the Fourier coefficient?\\ YL: yes, fixed.} 
\begin{equation}
  \label{eq:afm-bound}
  |a_f(m)| \ll_{\epsilon, \kappa}
N^{1/2 + \epsilon}
\lp  \sum_{n\leq c_2N^{2+\epsilon}}\vert a_f(n)\vert^2n^{1-\kappa}\rp^{1/2}
m^{\frac{\kappa-1}2 + \epsilon}
\end{equation}
for all $m$. 

\end{thm}
\begin{proof}
Let $Y>0$ be a small parameter chosen later. We start with the estimate
\begin{align}
	\Vert f\Vert^2 &= \frac{1}{[\textrm{SL}_2(\mathbb{Z})\colon \Gamma_0(N)]}\int_{\mathcal{F}_N} \vert f(x+iy)\vert^2y^{\kappa-2}dxdy \nonumber\\
	&\leq \frac{1}{[\textrm{SL}_2(\mathbb{Z})\colon \Gamma_0(N)]}\cdot I_f(Y) +  \frac{1}{[\textrm{SL}_2(\mathbb{Z})\colon \Gamma_0(N)]}\int_{\mathcal{F}_N\cap \mathcal{P}^c_Y} \vert f(x+iy)\vert^2y^{\kappa-2}dxdy \nonumber\\
	&\leq \frac{1}{[\textrm{SL}_2(\mathbb{Z})\colon \Gamma_0(N)]}\cdot I_f(Y) +  \frac{\textrm{Vol}(\mathcal{F}_N\cap \mathcal{P}^c_Y)}{[\textrm{SL}_2(\mathbb{Z})\colon \Gamma_0(N)]}\cdot \Vert f\Vert_{\infty}^2.\nonumber
\end{align}

According to \eqref{eq:sup_norm} and \eqref{eq:vol_est_cusps} we obtain
\begin{equation}
	\frac{\textrm{Vol}(\mathcal{F}_N\cap \mathcal{P}^c_Y)}{[\textrm{SL}_2(\mathbb{Z})\colon \Gamma_0(N)]}\cdot \Vert f\Vert_{\infty}^2 \ll_{\kappa,\epsilon} YN^{2+\epsilon}  \Vert f\Vert^2. \nonumber
\end{equation}
Let us call the implicit constant in this bound $C_{2} = C_2(\epsilon,\kappa)$. Similarly we write
\begin{equation}
	\frac{1}{[\textrm{SL}_2(\mathbb{Z})\colon \Gamma_0(N)]}\cdot I_f(Y) \leq \frac{C_1}{[\textrm{SL}_2(\mathbb{Z})\colon \Gamma_0(N)]}S_f(Y^{-1-\epsilon}) + C_1YN^{1+\epsilon}\Vert f\Vert^2,  \nonumber
\end{equation}
%\todo{TW: Is it $S_f(Y^{-1-\epsilon})$? EA: Yes thanks.}
for a constant $C_1=C_1(\epsilon,\kappa)$. This leaves us with
\begin{equation}
	\Vert f\Vert^2\leq \frac{C_1}{[\textrm{SL}_2(\mathbb{Z})\colon \Gamma_0(N)]}\cdot \left(1-C_2YN^{2+\epsilon}-C_1YN^{1+\epsilon}\right)^{-1}\cdot \sum_{n\leq Y^{-(1+\epsilon)}}\vert a_f(n)\vert^2n^{1-\kappa}.\nonumber
\end{equation}
Choosing $Y\asymp N^{-2-\epsilon}$ yields
\begin{equation}
	\Vert f\Vert^2 \ll_{\kappa,\epsilon} \frac{1}{[\textrm{SL}_2(\mathbb{Z})\colon \Gamma_0(N)]}\cdot \sum_{n\ll_{\kappa,\epsilon} N^{2+\epsilon}}\vert a_f(n)\vert^2n^{1-\kappa}.\nonumber
\end{equation}
This gives the desired bound on the Petersson norm. Substituting this into \eqref{eq:schulze-pillot} finishes the proof.
\end{proof}

\begin{rmk}
This argument is quite flexible and works for more general congruence subgroups (i.e. $\Gamma_1(N)$) and more general weights (i.e. $\kappa$ half integral).   
\end{rmk}

\begin{rmk}\label{rmk:length}
Considering that the dimension $d$ of $S_{\kappa}(N)$ satisfies $N^{1-\epsilon}\ll_{\kappa,\epsilon} d \ll_{\kappa,\epsilon}N^{1+\epsilon}$, the best bound for the Petersson norm of a general element $f\in S_{\kappa}(N)$ one can expect is of the form 
\begin{equation}
	\Vert f\Vert^2 \leq \frac{c_1}{[\mathrm{SL}_2(\mathbb{Z})\colon \Gamma_0(N)]}\cdot\sum_{n\leq c_2N^{1+\epsilon}}\vert a_f(n)\vert^2n^{1-\kappa}. \nonumber
\end{equation}
However, such a bound seems to be out of reach of our method.
\end{rmk}

\section{Proof of Theorem \ref{thm:main}}
\label{sec:pf-thm}
In this section, we will prove Theorem \ref{thm:main}.
To do this, we record two lemmas concerning the sizes of the Fourier coefficients of $g_{k, \pf, r}$ defined in \eqref{eq:gkr}.

\begin{lemma}
  \label{lemma:bd1}
  Let $g_{k, p, r} \in S_{2k}(N)$ be the modular form defined in \eqref{eq:gkr} with $k > 1, N = p^{1+2r}, r \ge 0$ and $p \in\Nb$ a prime.
  %%%with a factor $\pf \mid p$ a prime of $F$ inert in $K$.
  Then
  %%%there exists a constant $C > 1$ that only depends on $k$ and $D_i$ such that
  %there exists $C_1(k, \epsilon) > 0$ such that
  \begin{equation}
    \label{eq:bd1}
    |c_k(m; N)| \ll_{k, \epsilon}
    \frac{m^{k-\frac12+\epsilon}}{N\sqrt{D}}
    \min\lp
m^{1/2}, N^{5/2+\epsilon}
    \rp
\end{equation}
  for any $m \ge 1$ and $\epsilon > 0$.
  In particular, the implied constant is independent of $N$ and $m$.
\end{lemma}
%Combining the previous two bounds yields \eqref{eq:bd1}. 
\begin{proof}
  We only need to consider those primes $p$ such that there exists a prime $\pf \mid p$ of $F$ that is inert in $K$, otherwise $g_{k, p, r}$ is trivial.
  In that case any prime $\tilde\pf \mid p$ will be non-split in $K$.
For such $p$, putting together \eqref{eq:ctr} and \eqref{eq:ckr} then gives us
  $$
      c_k(n; p^{2r+1}) =      n^{k-1}
      \sum_{\pf \mid p}
      \sum_{a \in 2\Zb + nD, |a| < n\sqrt{D}}
      P_{k-1}\lp \frac{a}{n\sqrt{D}} \rp        \rho_{K/F}\lp\frac{n\sqrt{D} + a}2/\pf^{2r+1}\rp,      
%%%      \sum_{t \in \df^{-1}, t \gg 0, \tr(t) = n}
%%%      P_{k-1}\lp \frac{t - t'}{n} \rp        \rho_{K/F}(t\sqrt{D}/\pf^{2r+1}),
      $$
      where we have written $t = \frac{n \sqrt{D} + a}{2\sqrt{D}}$. 
      For totally positive
      $t_1, t_2 \in \pf^{2r+1}$ with
      $t_i = \frac{m \sqrt{D} + a_i}{2\sqrt{D}} \in \df^{-1}$, we have      $a_1 - a_2 \in \pf^{2r+1}\sqrt{D} \cap \Zb = p^{2r+1}D \Zb$.
      So for each $r$, any non-zero $a$ appearing in the summand is bounded above by $n/(p^{2r+1}\sqrt{D})$. Moreover, when $|a|<n\sqrt{D}$, we have $      P_{k-1}\lp \frac{a}{n\sqrt{D}} \rp < 1$. Finally, the ideal counting function $\rho_{K/F}(\af)$ is bounded by $\log \Nm(\af)$ for any integral ideal $\af$. Putting these together gives us
      \begin{equation}
        \label{eq:gkrn-bd1}
|c_k(n; p^{2r+1}) | \ll \frac{n^{k}\log n}{p^{2r+1}\sqrt{D}}.         
      \end{equation}
      Putting this into \eqref{eq:afm-bound} gives us
      %\todo[color=yellow]{EA: Included
        $n^{1-\kappa}$ in the sum %of \eqref{eq:afm-bound}. Does this change this estimate?\\ YL: Theorem \ref{thm:main} improved.}
      \begin{equation}
        \label{eq:gkrn-bd2}
        |c_k(m; p^{2r+1}) | \ll_{\epsilon}
        N^{3/2+\epsilon} D^{-1/2} m^{k-1/2 + \epsilon}. 
      \end{equation}
    \end{proof}

    \begin{proof}[Proof of Theorem \ref{thm:main}]
%%    \todo{YL: Theorem number in ref. Probably Theorem 1.1. EA: I think it is ok because we apply the Hecke operator just to one coordinate of $G_k$.}
      By equidistribution of Hecke correspondence on product of modular curves \cite[Theorem~1.1]{COU01} and negativity of higher Green function $G_k$, we have
      $$
|G^m_k(z_1, z_2)| \gg_{z_1, z_2, k} m
      $$
as $m \to \infty$. 
Applying this to \eqref{eq:Gm-est} with $k = 3, 5, 7$ and combining with Remarks \ref{rmk:r-bound} and Corollary \ref{cor:1}  gives us
\begin{equation}
  \label{eq:lower-bound}
  \sum_{ \substack{p, r\\p^{2r+1} \le \frac{m^2}4}}
  \log p
\cdot  |c_{k}(m; p^{2r+1}) |
=  \sum_{ \substack{p\in\pi(m), r\\p^{2r+1} \le \frac{m^2}4}}
  \log p
\cdot  |c_{k}(m; p^{2r+1}) |
\gg_{D_1, D_2} m^{k}.
\end{equation}
The size of the summation above is then bounded by $|\pi(m)| \cdot \log m$. 
% Assume that $\pi(m) \ll m^{\delta'}$ for some $\delta' < \delta := \frac1{5}$. Then there are at most $O_\epsilon(m^{\delta'+\epsilon})$ terms in the summation above by Corollary \ref{cor:1}.
%\iffalse\todo[color=yellow]{EA: Instead of assuming $\pi(m)\ll m^{\delta'}$ we can probably directly prove $m^k\ll \pi(m)\cdot m^{k-\frac{1}{5}+\epsilon}$. This should also resolve the issue arising from 'assuming the contrary'.\\ TW: Fix $\delta'<1/5$. I think if we assume there is an increasing sequence of $m_k$ such that $\pi(m_k)\leq C m_k^{\delta'}$ for some constant $C>0$ (independent of $k$), then we will get a contradiction. The condition  is equivalent to  $\limsup_{m\to \infty}\frac{\pi(m)}{m^{\delta'}}\leq C$. So the opposite should be $\limsup_{m\to \infty}\frac{\pi(m)}{m^{\delta'}}=\infty$. In particular, $\pi(m)\gg m^{\delta'}$. \\YL: modified the proof to directly lower bound $|\pi(m)|$}\fi
Denote $\delta = 1/5$. 
%%For $p^{2r+1}$ small,
On the one hand, the second bound in \eqref{eq:bd1} gives us
%\todo[inline, color=yellow]{EA: After the $c_k(m;p^{2r+1})$ have been estimated one could maybe record the non-vanishing condition $p\mid ...$ in the sum.\\ YL: modified (4.4) accordingly. EA: Thx}
  \begin{align*}
  \sum_{ \substack{p\in \pi(m), r\\p^{2r+1} \le m^\delta}}
    \log p
    \cdot  |c_{k}(m; p^{2r+1}) |
    & \ll_\epsilon
      |\pi(m)|\cdot \log m \cdot
      m^{k-1/2 + \epsilon} (m^\delta)^{3/2 + \epsilon}/\sqrt{D}.
    % & \ll_\epsilon
    %   |\pi(m)|
    %   m^{k-1/5 + \epsilon}/\sqrt{D}
  \end{align*}
  % \begin{align*}
  % \sum_{ \substack{p^{2r+1} \le m^\delta\\  \text{ power of a prime}\\ p \in \pi(m)}}
  %   \log p
  %   \cdot  |c_{k}(m; p^{2r+1}) |
  %   & \ll_\epsilon
  % \sum_{ \substack{p^{2r+1} \le m^\delta\\  \text{ power of a prime}}}
  %     m^{k-1/2 + \epsilon} 
  %     (p^{2r+1})^{3/2 +\epsilon}\\
  %   & \ll_\epsilon m^{k-1/2 + \epsilon} m^{\delta(1/(2\delta) - 1)+\epsilon} m^{\delta' + \epsilon}
  %      \ll_\epsilon m^{k-(\delta - \delta') +3 \epsilon}. 
  % \end{align*}
  On the other hand, the first bound in \eqref{eq:bd1} yields
  \begin{align*}
  \sum_{ \substack{p\in \pi(m), r\\ m^\delta < p^{2r+1} \le m^2/4}} \log p
    \cdot  |c_{k}(m; p^{2r+1}) |
    & \ll_\epsilon
      |\pi(m)| \cdot \frac{m^{k+\epsilon}}{m^\delta\sqrt{D}}.
%       \ll_\epsilon m^{k-(\delta - \delta') + \epsilon}. 
  \end{align*}
  
%   \begin{align*}
%   \sum_{ \substack{m^\delta < p^{2r+1} \le m^2/4\\  \text{ power of a prime}}} \log p
%     \cdot  |c_{k}(m; p^{2r+1}) |
%     & \ll_\epsilon
%       \sum_{\substack{m^\delta < p^{2r+1} \le m^2/4\\  \text{ power of a prime},~ p \in \pi(m)}}
%       \frac{m^{k + \epsilon} }{p^{2r+1}}\\
%     & \ll_\epsilon m^{k + \epsilon} m^{-\delta} m^{\delta' + \epsilon}.
% %       \ll_\epsilon m^{k-(\delta - \delta') + \epsilon}. 
%   \end{align*}
  Adding them together and putting it into \eqref{eq:lower-bound}
%%  Taking $k = 3$ then
  completes the proof.
    \end{proof}

\begin{rmk}\label{rmk:limit_of_method}
The quality of the lower bound in Theorem~\ref{thm:main} directly depends on the strength of the upper bound in \eqref{eq:bd1}. Indeed, running our argument with a general bound of the form
\begin{equation}
    |c_k(m; N)| \ll_{k, \epsilon}
    \frac{m^{k-\frac12+\epsilon}}{N\sqrt{D}}
    \min\lp
m^{1/2}, N^{\alpha+\epsilon}
    \rp \text{ for }\alpha\geq 1, \nonumber
\end{equation}
would yield
\begin{equation}
    \# \pi(m)\gg_{E_1,E_2,\epsilon} m^{\frac{1}{2\alpha}-\epsilon}.\nonumber
\end{equation}
In particular, the bound stated in Remark~\ref{rmk:length} would allow us to use $\alpha=\frac{3}{2}$ and thus produce the lower bound $\#\pi(m)\gg_{E_1,E_2,\epsilon} m^{\frac{1}{3}-\epsilon}$. We conclude that, using our method, any improvement upon the exponent $\frac{1}{3}$ would rely on a finer understanding of the Petersson norms of the cusp forms $g_{k,p,r}$.
\end{rmk}

\section{Relation with minimal isogeny degree bound}\label{sec: minimal isogeny degree}

We keep the notation from the introduction. Recall that for each $i = 1, 2$,  $E_i/H$ is a CM elliptic curve with CM by different orders $\Oc_{D_i}$. According to our discussion in the introduction and by Deuring's criterion, there is a density 1/4 subset of the rational primes $p$ satisfying that, for any $\pf$ in $H$ above $p$ that is also a prime of good reduction, one has $E_{1, \pf}\sim E_{2, \pf}$ over $\overline{\F}_{p}$.   For any such prime $\pf$ in $H$,   we define %\todo[color=yellow]{EA: I have rewritten the sentence slightly. Feel free to change things back if you prefer.\\TW: Look good. Thanks!}
\[
m_{D_1, D_2}(\pf) := \min\{\deg \phi: \phi: E_{1, \pf}\to E_{2, \pf} \text{ is a non-isomorphic  isogeny over $\overline{\F}_{p}$}\},
\]
which represents the minimal isogeny degree between $E_{1, \pf}$ and $E_{2, \pf}$. Note that $m_{D_1, D_2}(\pf)$ is independent of the choice of $\pf$, so we can also write it as $m_{D_1, D_2}(p)$. For primes $\mathfrak{p}$ where no isogeny exists between $E_{1, \pf}$ and $E_{2, \pf}$, we set $m_{D_1, D_2}(p) = 0$.

%%%between $E_{1, \pf}$ and $E_{2, \pf}$.
It is known that for supersingular elliptic curves, the minimal cyclic isogeny degree is at most $\lfloor \frac{1}{2} p^{2/3} + \frac{1}{4} \rfloor$ \cite[Section 4]{Elkies1987}, and the exponent $2/3$ is optimal \cite[Proposition 1.4, p. 1319]{Yang2008}. From the discussion before Theorem~\ref{thm:main}, for each prime $\mathfrak{p}$ in $H$, if $E_{1, \mathfrak{p}}$ and $E_{2, \mathfrak{p}}$ are isogenous, then they must be supersingular elliptic curves. In particular, the minimal isogeny bound applies to $E_{1, \mathfrak{p}}$ and $E_{2, \mathfrak{p}}$. The following result shows that either we can find better upper bounds of $m_{D_1, D_2}(p)$ for plenty of  primes $p$, or better upper bounds of $\#\pi(m)$ for almost all integers $m$.

\begin{prop}\label{prop:application-2}
Let $D_1$ and $D_2$ be coprime fundamental discriminants  and set $D:=D_1D_2$.
Let $x$  be a positive real number. For any $C>0$, $0<\delta<\frac{2}{3}$, and  $0<\eta<\frac{1}{2}$, as $x\to \infty$, we have either
\begin{enumerate}
    \item $\#\{p\leq x: m_{D_1, D_2}(p)\leq p^{2/3-\delta}\} \gg_{E_1, E_2} x^{1/3+\delta}$ or
    \item  $\#\{m\leq x: \#\pi(m) \geq C m^{1-\eta}\}\ll_{E_1, E_2, \eta, C} x^{\frac{5/6+\delta}{1-\eta}}$.
\end{enumerate}
In particular, for any sufficiently small $\epsilon>0$, we have either
\begin{enumerate}
    \item $\#\{p\leq \frac{x^2D}{2}: m_{D_1, D_2}(p)\leq p^{2/3-\epsilon}\} \gg_{E_1, E_2, \epsilon} x^{2/3+2\epsilon}$ or
    \item  $\#\{m\leq x: \#\pi(m) \geq C m^{\frac{5}{6}+2\epsilon}\}\ll_{E_1, E_2, \epsilon, C} x^{1-\epsilon}$.
\end{enumerate}

\end{prop}
\begin{proof}
First, we observe that
\begin{align}\label{cor:lower-bound}
\sum_{m \leq x}\#\pi(m)= \sum_{m \leq x}\#\{p\leq \frac{m^2D}{2}: \exists \;  \pf \text{ of $H$ such that } \pf\mid p, \mathfrak{p}\nmid N_{E_1}N_{E_2}, E_{1, \mathfrak{p}} \sim_m E_{2, \mathfrak{p}} \}, 
\end{align}
where $N_{E_1}$ and $N_{E_2}$ are conductors of $E_1$ and $E_2$, respectively.

On the other hand, by Deuring's correspondence, the equivalence class of a degree $m$ isogeny in $\Hom(E_{1, \mathfrak{p}}, E_{2, \mathfrak{p}})$ corresponds to a left ideal class in the endomorphism ring $\End_{\overline{\F}{p}}(E_{1, \mathfrak{p}})$ whose reduced norm is $m$. Therefore, the degree of any isogeny in $\Hom(E_{1, \frak{p}} , E_{2, \frak{p}})$  is of the form  $k m'$ for some $k\in \mathbb{N}$ and $m'\geq m_{D_1, D_2}(p)$. Hence, for any fixed $0< \delta< 2/3$, we have%\todo{EA: Have we defined $N_{E_1}$ and $N_{E_2}$?\\TW: Added the definition below (5.2). Thanks!}
\begin{align}\label{eq:upper-bound}
&  \sum_{m \leq x}\#\{p\leq \frac{m^2D}{2}: \exists \;  \pf \text{ of $H$ such that } \pf\mid p, \frak{p}\nmid N_{E_1}N_{E_2}, E_{1, \frak{p}} \sim_m E_{2, \frak{p}} \}\\
& = \sum_{p\leq \frac{x^2D}{2}}\#\left\{\sqrt{\frac{2p}{D}}\leq m \leq x:  \exists \;  \pf \text{ of $H$ such that } \pf\mid p,  \frak{p}\nmid N_{E_1}N_{E_2},   E_{1, \frak{p}} \sim_m  E_{2, \frak{p}}\right\} \nonumber \\
 & \ll_{E_1, E_2}  \sum_{p\leq \frac{x^2D}{2}}\frac{x}{m_{D_1, D_2}(p)} \nonumber \\
 &  \ll_{E_1, E_2} \sum_{\substack{p\leq \frac{x^2D}{2}\\ p^{2/3-\delta}\leq m_{D_1, D_2}(p)\leq p^{2/3}}} \frac{x}{m_{D_1, D_2}(p)}+ \sum_{\substack{p\leq \frac{x^2D}{2}\\ m_{D_1, D_2}(p)\leq p^{2/3-\delta}}} \frac{x}{m_{D_1, D_2}(p)} \nonumber \\
 &  \ll_{E_1, E_2} \sum_{\substack{p\leq \frac{x^2D}{2}\\ p^{2/3-\delta}\leq m_{D_1, D_2}(p)\leq p^{2/3}}} \frac{x}{p^{2/3-\delta}}+ \sum_{\substack{p\leq \frac{x^2D}{2}\\ m_{D_1, D_2}(p)\leq p^{2/3-\delta}}} x \nonumber\\
  &  \ll_{E_1, E_2} x^{5/3+2\delta}+ x \cdot \#\{p\leq \frac{x^2D}{2}: m_{D_1, D_2}(p)\leq p^{2/3-\delta}\}  \nonumber
\end{align}
We assume $ \#\{p\leq \frac{x^2D}{2}: m_{D_1, D_2}(p)\leq p^{2/3-\delta}\} \ll_{E_1, E_2} x^{2/3+2\delta}$.
% By Remark \ref{rmk:upper-bound-pi}, 
For any  $0<\eta<1/2$ and $C>0$ set \[M_{\eta}(x):=\#\{m\leq x: \#\pi(m) \geq C m^{1-\eta}\}
\] 
Combining with (\ref{cor:lower-bound}),  we obtain 
\begin{align*}
x^{5/3+2\delta} \gg_{E_1, E_2} \sum_{m\leq x}\#\pi(m)%\gg_\epsilon  \sum_{m\leq x} m^{1/5-\epsilon}\gg x^{6/5-\epsilon}.
= & \sum_{\substack{m\leq x\\ \#\pi(m)\geq C m^{1-\eta}}}\#\pi(m)+ \sum_{\substack{m\leq x\\ \#\pi(m)< C m^{1-\eta}}}\#\pi(m) \\
   \gg_{E_1, E_2} & \sum_{m\leq M_{\eta}(x)}  C m^{1-\eta}\gg_{E_1, E_2} M_{\eta}(x)^{2-2\eta}.
\end{align*}
Hence, 
\[
\#\{m\leq x: \#\pi(m) \geq C m^{1-\eta}\}\ll_{E_1, E_2, \eta, C} x^{\frac{5/6+\delta}{1-\eta}}.
\]
The first statement follows.

Now, for any sufficiently small $\epsilon>0$, taking $\delta=\epsilon$, $\eta=1/6-2\epsilon$, 
%$\eta=\frac{1}{12}-\epsilon (>0)$,  $\delta=\frac{1}{6}-2\eta (>0)$,  
%$\epsilon=1/3-\epsilon'$ for a sufficiently small $\epsilon'>0$ and   $\delta=2/3-2\epsilon=2\epsilon'$, 
as $x\to \infty$, 
we get either 
\[\#\{p\leq \frac{x^2D}{2}: m_{D_1, D_2}(p)\leq p^{2/3-\epsilon}\} \gg_{E_1, E_2, \epsilon} x^{2/3+2\epsilon}
\]
or 
\[
\#\{m\leq x: \#\pi(m) \geq C m^{\frac{5}{6}+2\epsilon}\}\ll_{E_1, E_2, \epsilon, C} x^{1-\epsilon}.
\]
The second result follows as well.
%\[\#\{p\leq \frac{x^2D}{2}: m_{D_1, D_2}(p)\leq p^{2/3-2\epsilon}\} \gg x^{2/3+4\epsilon}\]
%\[\#\{m\leq x: \#\pi(m) \geq C m^{\frac{11}{12}+\epsilon}\}\ll x^{\frac{10}{12}+3\epsilon}.\]
% \todo{JX: It seems that the right hand side becomes $x^{\frac{10}{11}- \epsilon}$ when we plug in  $\eta=\frac{1}{12}-\epsilon (>0)$ and  $\delta=\frac{1}{6}-2\eta (>0)$?\\
% TW: I  changed the choice of $\eta$ and $\delta$. Please let me know if therre are still question.  Thanks! EA: Thanks for renaming the parameters.}

%If $p$ is counted in both $\pi(m_1)$ and  $\pi(m_2)$ for some $m_1\neq m_2$, we get a self isogeny of $E_{1, \frak{p}}$ (or $E_{1, \frak{p}}$) of degree $m_1m_2$.   
\end{proof}

\printbibliography
%\bibliographystyle{amsalpha}
%\bibliography{CM-Isog}
%\bibliographystyle{amsplain}

\end{document}